\documentclass[11pt, a4paper]{amsart}
\usepackage[utf8]{inputenc}
\usepackage[english]{babel}

\usepackage{spectralsequences}
\usepackage[margin=2.5cm]{geometry}
\usepackage{mathrsfs}
\usepackage{enumitem}
\usepackage{pinlabel}
\usepackage{graphicx}
\usepackage{amsmath,amsfonts,amssymb,mathtools}
\usepackage{tikz,tikz-cd}
\usepackage{subcaption}
\usepackage{hyperref}
\usepackage{xcolor}
\usepackage{soul}
\usetikzlibrary{calc,patterns,arrows,shapes.arrows,intersections}
\usetikzlibrary{decorations}

\newtheorem{thm}{Theorem}[section]
\newtheorem{prop}[thm]{Proposition}
\newtheorem{lem}[thm]{Lemma}
\newtheorem{cor}[thm]{Corollary}
\newtheorem{rem}[thm]{Remark}
\newtheorem{example}[thm]{Example} 
\newtheorem{q}[thm]{Question}

\newtheorem{defn}[thm]{Definition}
 
\newtheorem*{question}{Question} 

\DeclareMathOperator{\Nerve}{\mathrm{N}}
\DeclareMathOperator{\Tot}{\mathrm{Tot}}
\DeclareMathOperator{\Fl}{\mathrm{Fl}}

\newcommand{\Uast}{\mathcal{U}^{\mathrm{ast}}}
\newcommand{\cU}{\mathcal{U}}
\newcommand{\Hh}{\mathrm{H}^h}
\newcommand{\UC}{\ddot{\mathrm{C}}}
\newcommand{\UH}{\ddot{\mathrm{H}}}
\newcommand{\hh}{\mathrm{H}}
\newcommand{\N}{\mathbb{N}}

\newcommand{\Z}{\mathbb{Z}}
\newcommand{\bH}{\mathbb{H}}

\newcommand{\cF}{\mathcal{F}}
\newcommand{\cH}{\mathcal{H}}
\newcommand{\cL}{\mathcal{L}}
\newcommand{\op}{\mathrm{op}}
\newcommand{\tG}{{\tt G}}

\newcommand{\tK}{{\tt K}}

\newcommand{\tT}{{\tt T}}
\newcommand{\tI}{{\tt I}}

\newcommand{\e}{\varepsilon}
\newcommand{\dd}{\delta}
\newcommand{\BH}{\ddot{\mathrm{B}}}

\definecolor{bluedefrance}{rgb}{0.19, 0.55, 0.91}
\definecolor{aquamarine}{rgb}{0.5, 1.0, 0.83}
\definecolor{princetonorange}{rgb}{1.0, 0.56, 0.0}
\definecolor{caribbeangreen}{rgb}{0.0, 0.8, 0.6}
\definecolor{bunired}{rgb}{0.8, 0.0, 0.0}
\definecolor{cdgreen}{rgb}{0.0, 0.42, 0.24}
\definecolor{lavender(floral)}{rgb}{0.71, 0.49, 0.86}

\title{From the Mayer-Vietoris spectral sequence to \"uberhomology}

\author{Luigi Caputi}
\author{Daniele Celoria}
\author{Carlo Collari}
\date{}

\begin{document}
\maketitle

\begin{abstract}
We prove that the second page of the Mayer-Vietoris spectral sequence, with respect to anti-star covers, can be identified with another homological invariant of simplicial complexes: the $0$-degree  \"uberhomology. Consequently, we obtain a combinatorial interpretation of the second page of the Mayer-Vietoris spectral sequence in this context. This interpretation is then used to extend the computations of bold homology, which categorifies the connected domination polynomial at~$-1$.
\end{abstract}

\section{Introduction}

The Mayer-Vietoris spectral sequence, a generalisation of the well-known Mayer-Vietoris long exact sequence, is an effective and far-reaching construction in algebraic topology. 
In its classical form, the Mayer-Vietoris spectral sequence is built from the intersection patterns of a cover of a given topological space~$X$.  This spectral sequence converges to the homology of $X$, thus providing a powerful tool in various  topological and combinatorial applications. First and foremost, 
it is related to the classical Nerve Lemma (\emph{cf}.~Theorem~\ref{thm:nervethm}).  
A proof of the Nerve Lemma using a  spectral sequence argument appeared in~\cite[Section~5]{segal}, and has  since been generalised in several ways. More recently, the Nerve Lemma and the Mayer-Vietoris spectral sequence have been  receiving increasing attention due to their applications in persistent homology~\cite{Lipsky,MR3857910,Casas,bauer2023unified,MR4537717}, bounded cohomology~\cite{ivanov2020leray, frigerio},   homology of configuration spaces~\cite{MR3597806,MR3797072,MR4200989}, and  topological complexity of spaces~\cite{Basu,daniel-david}. 
Classically, a description of the $E^2$-page was used to infer properties of spherical arrangements~\cite{MR1269311, MR1269312}, and to study hyperplane arrangements in general~\cite{MR1462732, davis2013vanishing, denham2016combinatorial}.

The Nerve Lemma holds for ``good covers'' of simplicial complexes (or, more generally, topological spaces), \emph{i.e.}~covers whose elements and  all of their possible finite intersections are contractible. When such intersections are not contractible, nor acyclic, the conclusion of the Nerve Lemma does not generally hold. Regardless, the Mayer-Vietoris spectral sequence eventually converges to the homology of $X$,  and the pages of the spectral sequence provide an increasingly accurate approximation of this convergence. Therefore, subcomplexes of the pages of the spectral sequence should contain interesting invariants -- \emph{e.g.}~other homology theories, as for path homology~\cite{asao}, or torsion, as in~\cite{MR1269311} -- which suggests the following  question:

\begin{question}
Let $\cU$  be a cover of a topological space $X$.
Which topological properties of $X$, or combinatorial properties of the nerve of $\cU$,  can be read from the second page of the Mayer-Vietoris spectral sequence?
\end{question}

As stated, this question is rather vague, as it strongly depends on both the cover $\cU$ and the space $X$.  In this paper we focus on the cover~$\cU^{\mathrm{ast}}$ consisting of anti-star subcomplexes~$\mathrm{ast}_X(v)$ of a finite, connected, simplicial complex $X$; here $v$ runs across the vertices of $X$, and $\mathrm{ast}_X(v)$ is the complement of the star of $v$ in $X$ (\emph{cf.}~Definition~\ref{defn:astarcomplex}).
Our main result is  that the second page of the Mayer-Vietoris spectral sequence of $X$, with respect to the cover~$\cU^{\mathrm{ast}}$,  is related to another combinatorial and homological invariant $\BH (X)$ of simplicial complexes: 

\begin{thm}\label{thm:uber=MV}
Let $X$ be a finite simplicial complex with $m$ vertices. Then, for all $ i\geq 0$ and $0\leq j\leq m$, there exists an isomorphism of bigraded modules 
\[
 E^2_{m-j-1, i} \cong \BH^j_i (X)
\]
between the second page of the augmented  Mayer-Vietoris spectral sequence of  ${X}$, associated to the anti-star cover, and the $0$-th degree \"uberhomology of $X$.  
\end{thm}
The bigraded homology $\BH^j_i (X)$ is the zero-degree specialisation of a more general homology theory of simplicial complexes introduced in~\cite{uberhomology}. The latter theory is called the  \emph{\"uberhomology} of $X$. It was shown in~\cite{uberhomology} that \"uberhomology contains both topological and combinatorial information about finite simplicial complexes. Theorem~\ref{thm:uber=MV} clarifies that, when setting  one of the degrees to~$0$, the \"uberhomology theory reads off combinatorial information of $X$ from the second page of the associated spectral sequence. 
It would be interesting to compare Theorem~\ref{thm:uber=MV} with the results of~\cite{everitt-turner-Booleancovers}. In that paper is presented a cellular-type cohomology associated to Boolean covers of a lattice. A Mayer-Vietoris spectral sequence is used to relate the homology of a Boolean cover to the homology of the underlying lattice. In their case, the cover corresponds to the upper intervals containing the atoms of the lattice, suggesting a parallel combinatorial interpretation of the Mayer-Vietoris spectral sequence in that context. 

Using the correspondence in Theorem~\ref{thm:uber=MV}, we provide some computations, and extend previous results of \cite{uberhomology,domination} first obtained via direct computations or discrete Morse theory techniques -- \emph{cf.}~Proposition~\ref{prop:gentrees} and Proposition~\ref{prop:cone}. We investigate the effect of coning and suspending. A further specialization of \"uberhomology, called bold homology, is related to connected dominating sets. We further prove that the $0$-degree \"uberhomology of flag complexes on triangle-free graphs is completely determined by their bold homology (Proposition~\ref{prop:triangle_free}).

Counting connected dominating sets in a graph is a NP-hard problem. It was shown in~\cite{domination} that bold homology yields a categorification of the connected domination polynomial.  
Denote by~$D_c(\tG)$ the connected domination polynomial of a graph $\tG$ (see Equation~\eqref{eq:cdompoly}).
The definition of the anti-star cover can be extended to regular CW-complexes, and our main application of Theorem~\ref{thm:uber=MV} is the following:

\begin{thm}\label{cor:chordalchar2}
Let $X\neq \Delta^n$ be a finite, connected, regular CW-complex. Assume that the anti-star cover of $X$ is a $1$-Leray cover, and that the $1$-skeleton of $X$ is a simple graph $\tG$. Then, we have 
\begin{equation*}
(-1)^{m-1}D_c(\tG)(-1)= \chi(X)-1 \ ,
\end{equation*}
where $m$ is the number of vertices in $X$.
\end{thm}

Theorem~\ref{cor:chordalchar2}, which is almost a straightforward consequence of Theorem~\ref{thm:uber=MV} in the case of simplicial complexes, is less direct when considering CW-complexes. For CW-complexes, a definition of \"uberhomology and bold homology groups is missing, and the bridge to dominating sets unclear. The role of Theorem~\ref{cor:chordalchar2} is to clarify this connection. For example, using the more general framework of CW-complexes, we are able to infer that $-1$ is a root of the connected domination polynomial of grid graphs (\emph{cf.}~Corollary~\ref{cor:grids}). 

Our last application concerns chordal graphs.
Recall that a chordal graph is a graph $\tG$ in which each induced cycle, \emph{i.e.}~a cycle that is an induced subgraph of $\tG$, has exactly three vertices. 
\begin{cor}\label{cor:chordalchar}
Let $\tG$ be a chordal graph on $m\geq 3$ vertices. Then, $-1$ is a root of the connected domination polynomial $D_c(\tG)$.
\end{cor}

Examples of non-chordal graphs $\tG$ for which $-1$ is a root of $D_c(\tG)$ are, for example, grid graphs (\emph{cf}.~Corollary~\ref{cor:grids}). In these cases,  the converse of Corollary~\ref{cor:chordalchar} does not hold.  
It would  be interesting to further explore the relation between topological properties of simplicial complexes and combinatorial properties of their $1$-skeleta, especially in relation with general notions of chordality~\cite{chordality}. 

\subsection*{Acknowledgements} 
The authors wish to warmly thank Julius Frank for valuable discussions, and for sharing his code~\cite{githububerjulius}. The authors are also grateful to Dejan Govc, for his  comments and feedbacks on the first draft of the paper, and wish to warmly thank the anonymous referee;  their
comments helped to improve the quality of the paper.  The main idea of this paper was born in a bar in Piazza Statuto in Turin.
DC was supported by Hodgson-Rubinstein’s ARC grant DP190102363 ``Classical And Quantum Invariants Of Low-Dimensional Manifolds''. CC is supported by the MIUR-PRIN project 2017JZ2SW5.

\section{\"Uberhomology, its 0-degree, and bold homology }\label{sec:uber}

We start by recalling the definition of \"uberhomology, following~\cite{uberhomology, domination}. 

Let $X$ be a finite and connected simplicial complex with $m$ vertices, say $V(X)=\{v_1,\dots,v_m\}$. In what follows, we assume that the vertices of $X$ are given a fixed order. The choice of such ordering is auxiliary and will not affect the following discussion. 

A \emph{bicolouring} $\varepsilon$ on $X$ is a map $\varepsilon\colon V(X) \to \{0,1\}$. As a visual aid, we will sometimes identify~$0$ with white and $1$ with black.
A \emph{bicoloured simplicial complex} is a pair $(X,\e)$, consisting of a simplicial complex $X$ equipped with the bicolouring~$\e$.
Given a $n$-dimensional simplex~$\sigma$ in $(X,\e)$, define its \emph{weight} with respect to $\e$ as the sum
\begin{equation}\label{eq:weight}
w_\varepsilon(\sigma )\coloneqq 
n+1-\sum_{v_i\in V(\sigma)} \varepsilon(v_i) \ .
\end{equation}
Equivalently, $w_\varepsilon(\sigma )$ is the number of $0$/white-coloured vertices in $\sigma$.
For a fixed bicolouring~$\e$, the weight in Equation~\eqref{eq:weight} induces a filtration of the simplicial chain complex~$C_*(X;\Z_2)$ associated to~$X$. 
More explicitly, set
\[ \mathscr{F}_j (X, \e )  \coloneqq \Z_2 \langle\ \sigma \mid w_{\e}(\sigma) \leq j\ \rangle \subseteq C_*(X;\Z_2) \ . \]
The simplicial differential $\partial$ respects this filtration; it can be decomposed as the sum of two differentials (\emph{cf}.~\cite[Lemma~2.1]{uberhomology}); one which preserves the weight, denoted by~$\partial_h$, and one which decreases the weight by one.   Call $(C(X,\e),\partial_h)$ the bigraded chain complex, whose underlying module is $C(X;\Z_2)$; the first degree is given by simplices' dimensions, while the weight $w_\e$ gives the second. 
The $\e$-\emph{horizontal homology} $\Hh(X,\e)$ of $(X,\e)$ is the homology of the bigraded chain complex $(C(X,\e),\partial_h)$. In other words, $\Hh(X,\e)$ is the homology of the graded object associated to the filtration $\mathscr{F}_j (X, \e )$. 

The next step towards the definition of the \"uberhomology is to note that the bicolourings on $X$ can be canonically identified with elements of the Boolean poset~$B(m)$ on a set with $m$ elements (partially ordered by inclusion).
Let $\e$ and $\e'$ be two bicolourings on $X$ differing only on a vertex $v_i$; assume further that $\e(v_i) = 0$ and~$\e'(v_i) =1$.  Denote by $d_{\e, \e'} $ the weight-preserving part of the identity map $\mathrm{Id} \colon \mathrm{H}^h(X,\e) \to \mathrm{H}^h(X,\e')$. With a slight abuse of notation, $d_{\e, \e'} $ can be written as
\[ d_{\e, \e'} (\sigma) = \begin{cases} \sigma & \text{if }w_{\e}(\sigma) = w_{\e'}(\sigma) \\ 0 &\text{otherwise}\end{cases} \ ,  \]
see~\cite[Section~6]{uberhomology}.
Note that the latter case can only occur if $w_{\e}(\sigma) = w_{\e'}(\sigma)-1$.

\begin{figure}[ht]
\begin{tikzpicture}[scale = .8]
\draw[dotted] (0,0) -- (0,-4.5) node[below] {$\UC^0(X)$};

\draw[dotted] (5,3) -- (5,-4.5) node[below] {$\UC^1(X)$};

\draw[dotted] (10,3) -- (10,-4.5) node[below] {$\UC^2(X)$};

\draw[dotted] (15,0) -- (15,-4.5) node[below] {$\UC^3(X)$};

\node[below] (uc0) at (0,-4.5) {\phantom{$\UC^0(X)$}};
\node[below] (uc1) at (5,-4.5) {\phantom{$\UC^0(X)$}};
\node[below] (uc2) at (10,-4.5) {\phantom{$\UC^0(X)$}};
\node[below] (uc3) at (15,-4.5) {\phantom{$\UC^0(X)$}};

\draw[->] (uc0) -- (uc1) node[midway, above] {$d^0$};
\draw[->] (uc1) -- (uc2) node[midway, above] {$d^1$};
\draw[->] (uc2) -- (uc3) node[midway, above] {$d^2$};

\node[fill, white] at (5,1.5){$\oplus$};
\node[fill, white] at (5,-1.5){$\oplus$};
\node  at (5,1.5){$\oplus$};
\node  at (5,-1.5){$\oplus$};

\node[fill, white] at (10,1.5){$\oplus$};
\node[fill, white] at (10,-1.5){$\oplus$};
\node  at (10,1.5){$\oplus$};
\node  at (10,-1.5){$\oplus$};

\node[fill, white] at (0,0) {${\Hh(X,(0,0,0))}$};

\node[fill, white] at (5,3){${\Hh(X,(1,0,0))}$};
\node[fill, white] at (5,0){${\Hh(X,(0,1,0))}$};
\node[fill, white] at (5,-3){${\Hh(X,(0,0,1))}$};

\node[fill, white] at (10,3) {${\Hh(X,(1,1,0))}$};
\node[fill, white] at (10,0) {${\Hh(X,(1,0,1))}$};
\node[fill, white] at (10,-3) {${\Hh(X,(0,1,1))}$};

\node[fill, white] at (15,0) {${\Hh(X,(1,1,1))}$};

\node (a) at (0,0) {${\Hh(X,(0,0,0))}$};

\node (b1) at (5,3) {${\Hh(X,(1,0,0))}$};
\node (b2) at (5,0) {${\Hh(X,(0,1,0))}$};
\node (b3) at (5,-3){${\Hh(X,(0,0,1))}$};

\node (c1) at (10,3) {${\Hh(X,(1,1,0))}$};
\node (c2) at (10,0) {${\Hh(X,(1,0,1))}$};
\node (c3) at (10,-3) {${\Hh(X,(0,1,1))}$};

\node (d) at (15,0) {${\Hh(X,(1,1,1))}$};

\draw[thick, ->] (a) -- (b1) node[midway,above,rotate =31] {$d_{(*,0,0)}$}; 
\draw[thick, ->] (a) -- (b2) node[midway,above] {$d_{(0,*,0)}$}; 
\draw[thick, ->] (a) -- (b3) node[midway,below,rotate =-29] {$d_{(0,0,*)}$}; 

\draw[thick, ->] (b1) -- (c1) node[midway,above] {$d_{(1,*,0)}$}; 
\draw[thick, ->] (b1) -- (c2) node[midway,above left,rotate =-29] {$d_{(1,0,*)}$}; 

\draw[thick, ->] (b3) -- (c2) node[midway,below left,rotate =31] {$d_{(*,0,1)}$}; 
\draw[thick, ->] (b3) -- (c3) node[midway,below] {$d_{(0,*,1)}$}; 

\draw[line width = 5, white] (b2) -- (c1) ; 
\draw[line width = 5, white] (b2) -- (c3) ; 
\draw[thick, ->] (b2) -- (c1) node[midway,below left,rotate =31] {$d_{(*,1,0)}$} ; 
\draw[thick, ->] (b2) -- (c3) node[midway,above left,rotate =-29] {$d_{(0,1,*)}$}; 

\draw[thick, <-] (d) -- (c1) node[midway,above,rotate =-29] {$d_{(1,1,*)}$}; 
\draw[thick, <-] (d) -- (c2) node[midway,above] {$d_{(1,*,1)}$}; 
\draw[thick, <-] (d) -- (c3) node[midway,below,rotate =31] {$d_{(*,1,1)}$}; 
\end{tikzpicture}

\caption{The Boolean poset $B(3)$ with vertices decorated by the horizontal homologies of a simplicial complex with $3$ vertices, and its ``flattening'' to the \"uber chain complex.}
\label{fig:cubo}
\end{figure}
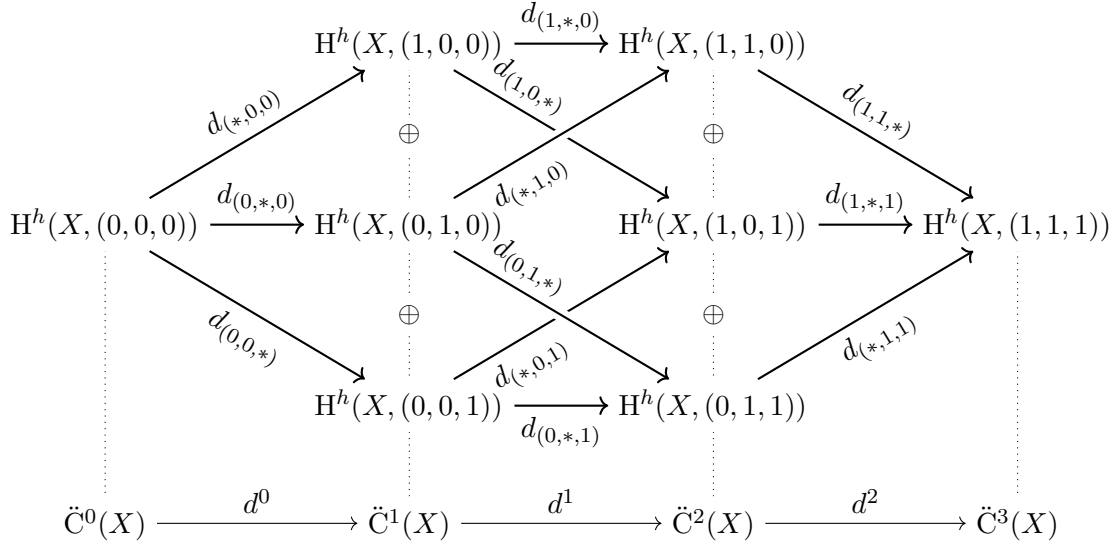

For a given bicolouring $\e$ on $X$,  set $\ell(\e) \coloneqq \sum_{j} \e(v_j) $.  The \emph{$j$-th \"uber chain module} is then defined as follows:
\begin{equation}\label{eq:uberchain}
\UC^{j}(X;\Z_2) = \bigoplus_{\ell(\e) = j} \Hh(X,\e) \ .
\end{equation}
By \cite[Proposition~6.2]{uberhomology}, the map 
\begin{equation}\label{eq:uberdiff}
\ddot{d}^j\coloneqq \sum_{\ell(\e) = j} d_{\e,\e'}\colon \UC^{j}(X;\Z_2)\to \UC^{j+1}(X;\Z_2)
\end{equation}
is a differential of degree $1$, turning $\left(\UC^{*}(X;\Z_2), \ddot{d}\right)$ into a cochain complex.
A schematic summary for the construction of the \"uber chain complex is presented in~Figure~\ref{fig:cubo}.

\begin{defn}\label{def:uber}
The \emph{\"uberhomology} $\UH^*(X)$ of a finite and connected simplicial complex $X$ is the homology of the complex $\left(\UC^{*}(X;\Z_2), \ddot{d}\right)$.
\end{defn}

\"Uberhomology groups can be endowed with two extra gradings, yielding a triply graded module. Indeed, the differential $\ddot{d}$ preserves both the simplices' weight and dimension. 
The notation for the three gradings on the \"uberhomology is as follows: $\UH^j_{k,i}(X)$ denotes the component of the homology generated by simplices of dimension $i$, with $k$ vertices of colour $0$, and whose (\"uber)homological degree is $j$.

The above definition can be rephrased in terms of poset homology~\cite{chandler2019posets, primo}, which allows extending the definition of \"uberhomology from $\Z_2$ to more general  coefficients:
\begin{prop}[{\cite[Proposition~2.14]{domination}}]\label{rem:ubergencoeff}
Let $X$ be a connected simplicial complex with $m$ vertices,  and $\mathbf{Mod}_{R}$ the category of $R$-modules over a commutative ring~$R$.
Then, the \"uberhomology of  $X$ with coefficients in $R$ coincides with the poset homology of $B(m)$, with coefficients in a suitable functor $\mathcal{H}\colon {B}(m)\to \mathbf{Mod}_{R}$. 
\end{prop}
We refer to \cite{domination} for the proof of Proposition~\ref{rem:ubergencoeff}, and for a more detailed account of the poset homology interpretation. 

As mentioned above, the \"uberdifferential $\ddot{d}$ preserves the $(k,i)$-bidegree. In particular,  specialising to the component of \"uberhomology of weight $0$ yields a bigraded homology.

\begin{defn}
For a simplicial complex $X$, define the $0$-degree \"uberhomology to be the bigraded homology
\[
\BH^j_i (X) \coloneqq \UH^j_{0,i}(X) \ .
\]
\end{defn}

An alternative definition of $\BH(X)$ is the following; for $\varepsilon \in \Z_2^m$ define $X_\varepsilon$ to be the  
simplicial subcomplex of $X$ induced by the $1$-coloured vertices with respect to $\varepsilon$. The homology $\BH^*_i(X)$ is obtained by decorating each vertex~$\varepsilon$ in the Boolean poset~$B(m)$ with the $i$-th homology~$ \mathrm{H}_i (X_\varepsilon)$ of $X_\e$; the differentials associated to the cube's edges are induced by inclusion. This is to say that $\BH^*_i(X)$ is the poset homology on the Boolean poset $B(m)$ with coefficients in the functor given by simplicial homology in dimension~$i$. 

Using Proposition~\ref{rem:ubergencoeff}, the definition of $\BH^*_*$ can be extended to encompass general coefficients; set $\BH^j_i(X;R)$ for the $0$-degree \"uberhomology of $X$, with coefficients in a commutative ring $R$.
\begin{example}
As an example, we can compute the homology $\BH^j_i(\partial \Delta^2)$ of the boundary of $\Delta^2$. The chain complex for is shown in Figure~\ref{fig:triang}. This complex is concentrated in homological degrees between $1$ and $3$, and simplicial degrees $0$ and $1$. In degree $i=0$, it is isomorphic to the simplicial chain complex associated to $\partial \Delta^{2}$, while in degree $i=1$ there are only trivial differentials, and a unique non-trivial summand in degree $j=3$. It follows that 
\[ \BH^{j}_{i}(\tK_3) = \begin{cases} \Z  & \text{if }(j,i)\in \{ (1,0),(3,1)\},\\ 0 & \text{otherwise.}\end{cases}\]
More explicitly, the generator in bidegree $(1,0)$ is spanned by the direct sum of the three connected components identified by a single black vertex (see the first column of Figure~\ref{fig:triang}). The other generator can instead be identified with the fundamental class of $\partial \Delta^2$, regarded as a triangulation of $S^1$. 

\begin{figure}[ht]

\begin{tikzpicture}[scale = .75]
\draw[dotted] (0,0) -- (0,-5.5) node[below] {\phantom{$\UC^0 (X)$}};

\draw[dotted] (5,3) -- (5,-5.5) node[below] {\phantom{$\UC^0 (X)$}};

\draw[dotted] (10,3) -- (10,-5.5) node[below] {\phantom{$\UC^0 (X)$}};

\draw[dotted] (15,0) -- (15,-5.5) node[below] {\phantom{$\UC^0 (X)$}};

\node[below] (uc0) at (0,-5.5) {{$_{\phantom{(0)}}0^{\phantom{(0)}}$}};
\node[below] (uc1) at (5,-5.5) {{$\Z_{(0)}^3$}};
\node[below] (uc2) at (10,-5.5) {{$\Z_{(0)}^3$}};
\node[below] (uc3) at (15,-5.5) {{$\Z_{(0)}^{\phantom{(0)}}\oplus \Z^{\phantom{(0)}}_{(1)}$}};

\draw[->] (uc0) -- (uc1) node[midway, above] {};
\draw[->] (uc1) -- (uc2) node[midway, above] {};
\draw[->] (uc2) -- (uc3) node[midway, above] {};

\node[fill, white] at (5,1.5){$\oplus$};
\node[fill, white] at (5,-1.5){$\oplus$};
\node  at (5,1.5){$\oplus$};
\node  at (5,-1.5){$\oplus$};

\node[fill, white] at (10,1.5){$\oplus$};
\node[fill, white] at (10,-1.5){$\oplus$};
\node  at (10,1.5){$\oplus$};
\node  at (10,-1.5){$\oplus$};

\node[fill, white] at (0,0) {${\rm H}_*\left(\text{\raisebox{-1em}{\begin{tikzpicture}[scale =1, very thick]
    \node (a) at (0,0) {} ;
    \node (b) at (1,0) {};
    \node (c) at (.5, .866) {};
    \draw[white] (0,0) circle (0.05) ;
    \draw[white] (1,0) circle (.05);
    \draw[white] (.5, .866) circle (.05);
    \draw[white] (a) -- (b);
    \draw[white] (c) -- (b);
    \draw[white] (a) -- (c);
    \end{tikzpicture}}}\right)$};

\node[fill, white] at (5,3){${\rm H}_*\left(\text{\raisebox{-1em}{\begin{tikzpicture}[scale =1, very thick]
    \node (a) at (0,0) {} ;
    \node (b) at (1,0) {};
    \node (c) at (.5, .866) {};
    \draw[, white] (0,0) circle (0.05) ;
    \draw[, white] (1,0) circle (.05);
    \draw[, white] (.5, .866) circle (.05);
    \draw[white] (a) -- (b);
    \draw[white] (c) -- (b);
    \draw[white] (a) -- (c);
    \end{tikzpicture}}}\right)$};
\node[fill, white] at (5,0){${\rm H}_*\left(\text{\raisebox{-1em}{\begin{tikzpicture}[scale =1, very thick]
    \node (a) at (0,0) {} ;
    \node (b) at (1,0) {};
    \node (c) at (.5, .866) {};
    \draw[, white] (0,0) circle (0.05) ;
    \draw[, white] (1,0) circle (.05);
    \draw[, white] (.5, .866) circle (.05);
    \draw[white] (a) -- (b);
    \draw[white] (c) -- (b);
    \draw[white] (a) -- (c);
    \end{tikzpicture}}}\right)$};
\node[fill, white] at (5,-3){${\rm H}_*\left(\text{\raisebox{-1em}{\begin{tikzpicture}[scale =1, very thick]
    \node (a) at (0,0) {} ;
    \node (b) at (1,0) {};
    \node (c) at (.5, .866) {};
    \draw[, white] (0,0) circle (0.05) ;
    \draw[, white] (1,0) circle (.05);
    \draw[, white] (.5, .866) circle (.05);
    \draw[white] (a) -- (b);
    \draw[white] (c) -- (b);
    \draw[white] (a) -- (c);
    \end{tikzpicture}}}\right)$};

\node[fill, white] at (10,3) {${\rm H}_*\left(\text{\raisebox{-1em}{\begin{tikzpicture}[scale =1, very thick]
    \node (a) at (0,0) {} ;
    \node (b) at (1,0) {};
    \node (c) at (.5, .866) {};
    \draw[, white] (0,0) circle (0.05) ;
    \draw[, white] (1,0) circle (.05);
    \draw[, white] (.5, .866) circle (.05);
    \draw[white] (a) -- (b);
    \draw[white] (c) -- (b);
    \draw[white] (a) -- (c);
    \end{tikzpicture}}}\right)$};
\node[fill, white] at (10,0) {${\rm H}_*\left(\text{\raisebox{-1em}{\begin{tikzpicture}[scale =1, very thick]
    \node (a) at (0,0) {} ;
    \node (b) at (1,0) {};
    \node (c) at (.5, .866) {};
    \draw[, white] (0,0) circle (0.05) ;
    \draw[, white] (1,0) circle (.05);
    \draw[, white] (.5, .866) circle (.05);
    \draw[white] (a) -- (b);
    \draw[white] (c) -- (b);
    \draw[white] (a) -- (c);
    \end{tikzpicture}}}\right)$};
\node[fill, white] at (10,-3) {${\rm H}_*\left(\text{\raisebox{-1em}{\begin{tikzpicture}[scale =1, very thick]
    \node (a) at (0,0) {} ;
    \node (b) at (1,0) {};
    \node (c) at (.5, .866) {};
    \draw[, white] (0,0) circle (0.05) ;
    \draw[, white] (1,0) circle (.05);
    \draw[, white] (.5, .866) circle (.05);
    \draw[white] (a) -- (b);
    \draw[white] (c) -- (b);
    \draw[white] (a) -- (c);
    \end{tikzpicture}}}\right)$};

\node[fill, white] at (15,0) {${\rm H}_*\left(\text{\raisebox{-1em}{\begin{tikzpicture}[scale =1, very thick]
    \node (a) at (0,0) {} ;
    \node (b) at (1,0) {};
    \node (c) at (.5, .866) {};
    \draw[, white] (0,0) circle (0.05) ;
    \draw[, white] (1,0) circle (.05);
    \draw[, white] (.5, .866) circle (.05);
    \draw[white] (a) -- (b);
    \draw[white] (c) -- (b);
    \draw[white] (a) -- (c);
    \end{tikzpicture}}}\right)$};

\node (a0) at (0,0) {{${\rm H}_*\left(\text{\raisebox{-1em}{\begin{tikzpicture}[scale =1, very thick]
    \node (a) at (0,0) {} ;
    \node (b) at (1,0) {};
    \node (c) at (.5, .866) {};
    \draw[fill, gray, opacity = .2] (0,0) circle (0.05) ;
    \draw[fill, gray, opacity = .2] (1,0) circle (.05);
    \draw[fill, gray, opacity = .2] (.5, .866) circle (.05);
    \draw[gray, opacity = .2] (a) -- (b);
    \draw[gray, opacity = .2] (c) -- (b);
    \draw[gray, opacity = .2] (a) -- (c);
    \end{tikzpicture}}}\right)$}};

\node (b1) at (5,3) {{${\rm H}_*\left(\text{\raisebox{-1em}{\begin{tikzpicture}[scale =1, very thick]
    \node (a) at (0,0) {} ;
    \node (b) at (1,0) {};
    \node (c) at (.5, .866) {};
    \draw[fill, black] (0,0) circle (0.05) ;
    \draw[fill, gray, opacity = .2] (1,0) circle (.05);
    \draw[fill, gray, opacity = .2] (.5, .866) circle (.05);
    \draw[gray, opacity = .2] (a) -- (b);
    \draw[gray, opacity = .2] (c) -- (b);
    \draw[gray, opacity = .2] (a) -- (c);
    \end{tikzpicture}}}\right)$}};
\node (b2) at (5,0) {{${\rm H}_*\left(\text{\raisebox{-1em}{\begin{tikzpicture}[scale =1, very thick]
    \node (a) at (0,0) {} ;
    \node (b) at (1,0) {};
    \node (c) at (.5, .866) {};
    \draw[fill, gray, opacity = .2] (0,0) circle (0.05) ;
    \draw[fill, black] (1,0) circle (.05);
    \draw[fill, gray, opacity = .2] (.5, .866) circle (.05);
    \draw[gray, opacity = .2] (a) -- (b);
    \draw[gray, opacity = .2] (c) -- (b);
    \draw[gray, opacity = .2] (a) -- (c);
    \end{tikzpicture}}}\right)$}};
\node (b3) at (5,-3){{${\rm H}_*\left(\text{\raisebox{-1em}{\begin{tikzpicture}[scale =1, very thick]
    \node (a) at (0,0) {} ;
    \node (b) at (1,0) {};
    \node (c) at (.5, .866) {};
    \draw[fill, gray, opacity = .2] (0,0) circle (0.05) ;
    \draw[fill, gray, opacity = .2] (1,0) circle (.05);
    \draw[fill, black] (.5, .866) circle (.05);
    \draw[gray, opacity = .2] (a) -- (b);
    \draw[gray, opacity = .2] (c) -- (b);
    \draw[gray, opacity = .2] (a) -- (c);
    \end{tikzpicture}}}\right)$}};

\node (c1) at (10,3) {{${\rm H}_*\left(\text{\raisebox{-1em}{\begin{tikzpicture}[scale =1, very thick]
    \node (a) at (0,0) {} ;
    \node (b) at (1,0) {};
    \node (c) at (.5, .866) {};
    \draw[fill, black] (0,0) circle (0.05) ;
    \draw[fill, black] (1,0) circle (.05);
    \draw[fill, gray, opacity = .2] (.5, .866) circle (.05);
    \draw[black] (a) -- (b);
    \draw[gray, opacity = .2] (c) -- (b);
    \draw[gray, opacity = .2] (a) -- (c);
    \end{tikzpicture}}}\right)$}};
\node (c2) at (10,0) {{${\rm H}_*\left(\text{\raisebox{-1em}{\begin{tikzpicture}[scale =1, very thick]
    \node (a) at (0,0) {} ;
    \node (b) at (1,0) {};
    \node (c) at (.5, .866) {};
    \draw[fill, black] (0,0) circle (0.05) ;
    \draw[fill, gray, opacity = .2] (1,0) circle (.05);
    \draw[fill, black] (.5, .866) circle (.05);
    \draw[gray, opacity = .2] (a) -- (b);
    \draw[gray, opacity = .2] (c) -- (b);
    \draw[black] (a) -- (c);
    \end{tikzpicture}}}\right)$}};
\node (c3) at (10,-3){{${\rm H}_*\left(\text{\raisebox{-1em}{\begin{tikzpicture}[scale =1, very thick]
    \node (a) at (0,0) {} ;
    \node (b) at (1,0) {};
    \node (c) at (.5, .866) {};
    \draw[fill, gray, opacity = .2] (0,0) circle (0.05) ;
    \draw[fill, black] (1,0) circle (.05);
    \draw[fill, black] (.5, .866) circle (.05);
    \draw[gray, opacity = .2] (a) -- (b);
    \draw[black] (c) -- (b);
    \draw[gray, opacity = .2] (a) -- (c);
    \end{tikzpicture}}}\right)$}};

\node (d) at (15,0) {{${\rm H}_*\left(\text{\raisebox{-1em}{\begin{tikzpicture}[scale =1, very thick]
    \node (a) at (0,0) {} ;
    \node (b) at (1,0) {};
    \node (c) at (.5, .866) {};
    \draw[fill, black] (0,0) circle (0.05) ;
    \draw[fill, black] (1,0) circle (.05);
    \draw[fill, black] (.5, .866) circle (.05);
    \draw[black] (a) -- (b);
    \draw[black] (c) -- (b);
    \draw[black] (a) -- (c);
    \end{tikzpicture}}}\right)$}};

\draw[thick, ->] (a0) -- (b1);
\draw[thick, ->] (a0) -- (b2); 
\draw[thick, ->] (a0) -- (b3);

\draw[thick, ->] (b1) -- (c1); 
\draw[thick, ->] (b1) -- (c2);

\draw[thick, ->] (b3) -- (c2);
\draw[thick, ->] (b3) -- (c3);
\draw[line width = 5, white] (b2) -- (c1) ; 
\draw[line width = 5, white] (b2) -- (c3) ; 
\draw[thick, ->] (b2) -- (c1);
\draw[thick, ->] (b2) -- (c3);

\draw[thick, <-] (d) -- (c1) ;
\draw[thick, <-] (d) -- (c2) ; 
\draw[thick, <-] (d) -- (c3) ;
\end{tikzpicture}

\caption{The $0$-degree \"uberchain complex of $\partial\Delta^2$. Here $\Z^{d}_{(i)}$ denotes a $\Z^d$ summand in $\BH^{*}_{i}$. }\label{fig:triang}
\end{figure}
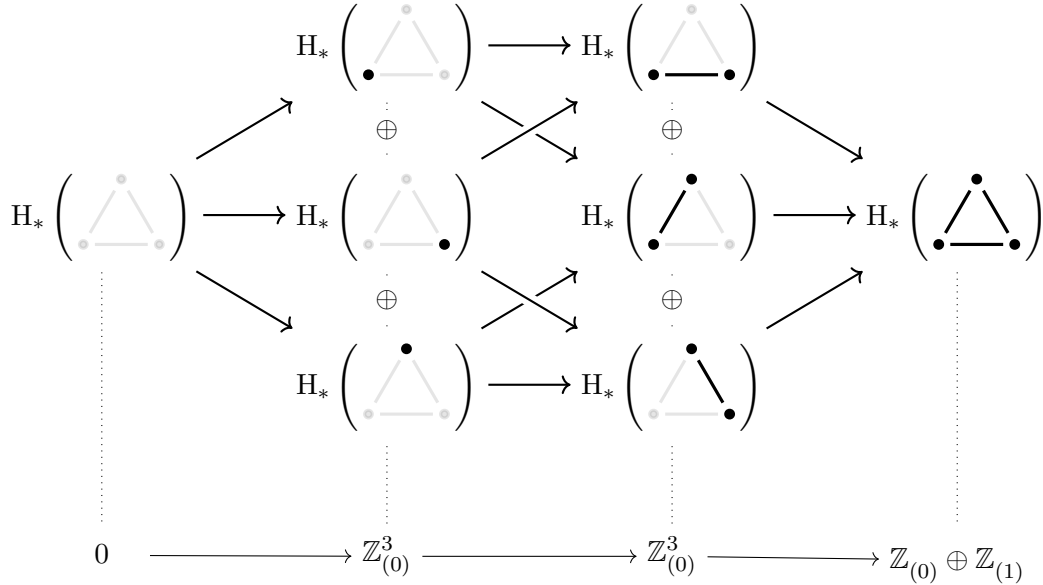
\end{example}

\subsection{Bold homology}

The specialisation of $\BH^*_i(X)$ to $i=0$ is known as \emph{bold homology}, and is denoted by $\mathbb{H}^* (X)$. This homology was introduced in~\cite[Section~8]{uberhomology}, and it was shown to contain non-trivial combinatorial information on simple  graphs~\cite{domination}. 
The relations between the three homologies introduced so far can be schematically summarised as follows:
$$\UH^j_{k,i}(X) \xrightarrow[k = 0]{\text{restrict to}} \BH^j_i(X) \xrightarrow[\text{i = 0}]{\text{restrict to}} \mathbb{H}^j(X) $$

Let $\tG$ be a connected simple graph, that is, a connected $1$-dimensional simplicial complex. A~subset $S\subseteq V(\tG)$ of the vertices of $\tG$ is called
\begin{itemize}
\item \emph{dominating} if each vertex in $V(\tG)$ either belongs to $S$ or shares an edge with some element of $S$;
\item \emph{connected} if $S$ spans a connected subcomplex.
\end{itemize}
The \emph{connected domination polynomial} of a graph $\tG$ is defined as 
\begin{equation}\label{eq:cdompoly}
    D_c(\tG)(t) = \sum_{S} t^{|S|} \, \in \Z[t] \ ,
\end{equation} 
where $S$ ranges among connected dominating sets in $V(\tG)$. Computing the connected domination polynomial of a graph is known to be NP-hard~\cite{garey1979computers}. In~\cite{domination}, the authors prove the existence of a tight relation between connected domination polynomials  and the bold homology's Euler characteristic:

\begin{thm}[{\cite[Theorem~1.2]{domination}}]
The bold homology  categorifies $D_c(-1)$. More precisely,~$\mathbb{H}^*$ is functorial under inclusion of graphs, and its Euler characteristic is $D_c(-1)$.
\end{thm}

From this result, some properties and computations of $\mathbb{H}$ can be deduced. For example, 
the bold homology of trees is zero, and it detects complete graphs (see~\cite{domination} for the precise statements). 
Computations performed with bold homology can be extended to simplicial complexes as well; indeed, the following can be deduced at once from the definitions:

\begin{lem}\label{lem:flaggraphs}
Let $X$ be a simplicial complex, and $X^{(1)}$ its $1$-skeleton. Then, there exists a graded isomorphism $\mathbb{H}^*(X) \cong \mathbb{H}^*(X^{(1)})$. 
\end{lem}

\begin{proof}
Connected components of $1$-coloured subcomplexes in $X^{(1)}$ are in canonical bijection with connected components in $X$. Then, the result follows by \cite[Theorem~1.3]{domination}.
\end{proof}

\section{Anti-star covers and spectral sequences}\label{sec:MVSS}

One of the main  tools employed in (co)homology computations is the generalization of the Mayer-Vietoris long exact sequence in terms of spectral sequences. To set the notations, we start by recalling some basic definitions, referring to~\cite{user} for further details.

We will focus on \emph{augmented} first quadrant spectral sequences of homological type, \emph{i.e.}~spectral sequences arising from first quadrant augmented bicomplexes, whose induced differentials have bidegree $(-r,r-1)$. 
By a first quadrant augmented bicomplex $(C_{p,q},\delta_{p,q})$ of bidegree $(a,b)$ we mean a bigraded $R$-module $C_{p,q}$ with differentials
 \[
  \delta_{p,q}\colon C_{p,q}\longrightarrow C_{p+a,q+b}  \ ,
 \]
where $C_{p,q}=0$ for $p<-1$ and $q<0$. Spectral sequences arise naturally in the context of filtered chain complexes. 
As customary, we say that a spectral sequence $(E^r,d^r)$ \emph{converges} to a graded module $\hh_*$, and write 
$E_{p,q}\Rightarrow \hh_{p+q}$, 
if there is a filtration $F$ on $\hh_*$ such that $E^{\infty}_{p,q}\cong \mathrm{Gr}_{p,q}\hh_*$ for all $p,q$, where $E^{\infty}_{p,q}$ is the limit term of the spectral sequence.

We now specialize  to first quadrant augmented bicomplexes~$(C,\dd_v,\dd_h)$ where $\dd_v$ and  $\dd_h$ are differentials of bidegrees $(0,-1)$ and $(-1,0)$, respectively, and $\dd_v\circ \dd_h=\dd_h\circ \dd_v$. 
Consider the associated \emph{total complex}  
\[\Tot(C)_n:=\bigoplus_{p+q=n}C_{p,q} \ ,\]
with differential $\dd$ defined
by setting $\dd (x):=\dd_h(x)+(-1)^p\dd_v(x)$ for each $x\in C_{p,*}$, and each~$p$. 
There are two natural filtrations  on $\Tot (C)$. The first filtration $F^I$ is defined by cutting the direct sum above at the $p$-level:
$F^I_p(\Tot (C))_n:= \bigoplus_{i\leq p} C_{i,n-i} $.
The second filtration $F^{II}$ is the complementary one:
$F^{II}_q(\Tot(C))_n:=\bigoplus_{j\leq q} C_{n-j,j} $.
In the special case of a first quadrant double chain complex, both filtrations  are bounded (from above and below), hence both the associated spectral sequences  converge to the homology  of the total complex~$\Tot(C)$. 
The $0$-page of the spectral sequence arising from the first filtration is \[
 IE^0_{p,q}:= F_p^I (\Tot (C))_{p+q}/ F_{p-1}^I (\Tot (C))_{p+q}=\bigoplus_{i\leq p} C_{i,p+q-i}/\bigoplus_{i\leq p-1} C_{i,p+q-i}=C_{p,q} \ ,
\]
and for the second filtration is:
\[
 IIE^0_{p,q}:= F_p^{II} (\Tot (C))_{p+q}/ F_{p-1}^{II} (\Tot (C))_{p+q}=\bigoplus_{j\leq p} C_{p+q-j,j}\slash\bigoplus_{j\leq p-1} C_{p+q-j,j}=C_{q,p} \ .
\]
The differentials are respectively induced by $\dd_v$ and by $\dd_h$.
The first pages of the associated spectral sequences are given by $IE^1_{p,q}=\hh_q(C_{p,*})
$, with induced differential~$\dd^{(2)}_h$, and by $IIE^1_{p,q}=\hh_q(C_{*,p})$, with induced differential~$\dd_v^{(2)}$, simply denoted by $\dd^{(2)}$ in the follow-up, respectively.

\subsection{The Mayer-Vietoris spectral sequence}

Following \cite[Chapter VII.4]{brown} and \cite[Sections I.3.3 \& II.5]{godement}, we briefly recall the construction of the Mayer-Vietoris spectral sequence.

For a simplicial complex $X$, denote by~$P(X)$ the face poset of $X$, \emph{i.e.}~the poset of non-empty simplices of~$X$, ordered by inclusion. 
Let $X_p$ be the set consisting of the $p$-simplices in $X$, and  assume that the set of vertices is always \emph{finite} and \emph{ordered}. The results will not depend on the choice of ordering. 

Let  $\cU=\{U_i\}_{i\in I}$ be a simplicial cover of~$X$, \emph{i.e.}~a family of subcomplexes of $X$ with the property that each $U_i$ is non-empty and $\bigcup_{i\in I}U_i=X$. For each non-empty subset~$J$ of the set of indices~$I$, denote by $U_{J}$ the intersection $\bigcap_{j\in J} U_j$ of the corresponding elements in the cover. From this data, we can associate to $X$ another simplicial complex:

\begin{defn}
Given a simplicial complex $X$ and cover $\cU=\{U_i\}_{i\in I}$, the \emph{nerve}~$\Nerve (\cU)$ is the simplicial complex on the family of non-empty finite subsets $J\subseteq I$, such that $U_{J}\coloneqq \bigcap_{j\in J} U_j \neq \emptyset$.
\end{defn}

Consider the cover $\cU=\{U_i\}_{i\in I}$ of $X$. We can associate to each simplex~$\sigma$ of~$\Nerve (\cU)$ a subset $J$ of $\{ 1, ..., \vert I\vert\}$. 
In particular, for a $p$-simplex $\sigma$ of $\Nerve (\cU)$, defined by the indices~$j_0, \dots, j_p$, we will denote by $U_\sigma$ the intersection $U_{j_0}\cap \dots \cap U_{j_p}$.

Every poset is a category. In particular, this holds for the face poset $P(X)$; its objects are the simplices of $X$, and there is a morphism $\tau\to \sigma$ whenever $\tau$ is a face of $\sigma$. Denote the category obtained this way by $\mathbf{P}(X)$, and by $\mathbf{Ab}$ the category of Abelian groups.

\begin{defn}
A \emph{coefficient system}  on $X$ is a functor~$\cL\colon \mathbf{P}(X)^\op\to \mathbf{Ab}$.
\end{defn}

More concretely, a coefficient system on $X$ is a family of Abelian groups $\{A_{\sigma}\}$, indexed by the simplices~$\sigma$ of $X$, together with a map $A_{\tau\subseteq \sigma}\colon A_{\sigma}\to A_{\tau}$ whenever $\tau $ is a face of $\sigma$, and such that $A_{\tau\subseteq \sigma}\circ A_{\mu\subseteq \tau}=A_{\mu\subseteq \sigma}$ if $\mu\subseteq \tau\subseteq \sigma$.

\begin{example}\label{ex:coeffsys}
Let $\cU$ be a  cover of a simplicial complex~$X$.  For each $q \in \N$  and $\sigma\in \Nerve(\cU) $, define $$\cH_q(\sigma)\coloneqq \hh_q(U_\sigma)$$ as the $q$-th homology group of the subcomplex $U_\sigma$ of $X$. If $\tau$ is a face of $\sigma$ belonging to $\Nerve(\cU)$, then there is an inclusion $U_{\sigma}\subseteq U_{\tau}$, hence an induced map between the associated $q$-homology groups. It is straightforward to check that $\cH_q$ yields a coefficient system on the nerve $\Nerve(\cU)$. Analogously, the group $C_q(U_\sigma)$ of $q$-chains in $U_\sigma$ can be considered; this also yields a coefficient system $\mathcal{C}_q$ on $\Nerve(\cU)$.  
\end{example}

Given a simplicial complex $X$ and a coefficient system~$\cL$ on $X$, it is possible to define the homology groups of $X$ with coefficients in $\cL$, \emph{cf}.~\cite[Section I.3.3]{godement}. For each $n\geq 0$, define the  $p$-chains as the sum
\[
C_p(X;\cL)\coloneqq \bigoplus_{\sigma \in X_p}\cL(\sigma)\ .
\]
If~$\sigma$ is the simplex $[x_0, \dots, x_p]$, then set $d_i(\sigma)\coloneqq [x_0, \dots, \widehat{x_i}, \dots, x_p]$ for its $i$-th face. By functoriality of $\cL$, there are restriction maps 
\[
\cL(\sigma)\to \cL(d_i(\sigma))\hookrightarrow \bigoplus_{\tau \in X_{p-1}}\cL(\tau)
\]
for all $\sigma\in X_p$ and $0\leq i\leq p$, extending to maps  
$\partial_i\colon C_p(X;\cL)\to C_{p-1}(X;\cL)$ on the whole chain complex $C_p(X;\cL)$, for each $i$. 
The differential $\partial\colon C_p(X;\cL)\to C_{p-1}(X;\cL)$ is defined  by setting $\partial \coloneqq \sum_{i=0}^p (-1)^i\partial_i$. This is usually known as \v{C}ech differential, and the resulting chain complex is usually called cosheaf complex (see \emph{e.g.}~\cite[Section~4]{bredon}).

\begin{defn}
The homology of $X$ with coefficients in the coefficient system $\cL$ is the homology of the chain complex $(C_*(X;\cL),\partial)$.
\end{defn} 

We now turn to the construction of the Mayer-Vietoris spectral sequence. This is the spectral sequence associated to a double chain complex, corresponding to a cover of a topological space. For a simplicial complex $X$ and cover $\cU$, set
\begin{equation}\label{eq:E0}
C^0_{p,q} \coloneqq \bigoplus_{\sigma\in \mathrm{N}_p(\cU)} C_q\left(U_\sigma\right)=\bigoplus_{|J|=p+1} C_q\left(U_J\right)
\end{equation}
for  the $\Z$-module 
freely generated by the $q$-chains of the subcomplexes of $X$ obtained considering intersections of $p+1$ elements of the cover~$\cU$. There are two differentials, decreasing either the~$p$ or the $q$ degree. Denote by $\delta^0_h\colon C^0_{p,q}\to C^0_{p-1,q}$ the horizontal differential decreasing the $p$-degree, and by $\delta^0_v\colon C^0_{p,q}\to C^0_{p,q-1}$ the vertical one decreasing the $q$-degree. 
As customary, we define the differential $\delta^0_v$ as the alternating sum 
over the faces: if $\tau=[v_0,\dots, v_q]$, then $\dd^0_v(\tau)\coloneqq \sum_{k=0}^q(-1)^k[v_0,\dots,\widehat{v_k},\dots, v_q]$. This way, for all $p\geq 0$, we get chain complexes $(C^0_{p,*},\dd^0_v)$. In the horizontal direction instead, at a fixed~$q\in\N$, we define $\dd_h^0$ as the differential of the chain complex $C_p(\Nerve(\cU); \mathcal{C}_q)$ of $\Nerve(\cU)$ with coefficients in the coefficient system $\mathcal{C}_q$. More concretely, for each $J$ appearing in the sum of Equation~\eqref{eq:E0}, and each $j\in J$, the subset $J'\coloneqq J\setminus\{j\}$ is a face of $J$ in $\Nerve(\cU)$. The inclusion of $J'$ in $J$ induces a simplicial map $U_{J}\to U_{J'}$ (reversing the ordering), hence a map on the level of $q$-chains. Then, the differential $\dd^0_h$ is defined on the basis elements of $C_q(U_J)$ as the alternating sum of $U_{J\setminus\{j\}}$ over $j \in J$. This definition is then extended to all sums by linearity. 

\begin{rem}
The differentials $\dd_h^0$ and $\dd_v^0$ commute, \emph{i.e.}~$\dd_h^0\circ \dd_v^0=\dd_v^0 \circ \dd_h^0$.
\end{rem}

Endowing the groups $C^0_{p,q}$ with the differentials $\dd_h^0$ and $\dd_v^0$, yields a double chain complex. Both spectral sequences associated to the double chain complex $(C^0_{*,*}, \dd_h^0, \dd_v^0)$ (corresponding to the vertical and horizontal filtration) converge to the homology of the total chain complex~$\Tot X$, since $(C^0_{*,*}, \dd_h^0, \dd_v^0)$ is a first quadrant double chain complex. 
However, even though the two spectral sequences abut to the same graded object $\hh_*(\Tot X)$, they have different $E^\infty$-terms -- seen as bigraded objects. First,  consider the spectral sequence $IE$ obtained by taking homology with respect to the $p$-degree. As, for $s>0$ fixed, the chain complexes~$C_{*,s}^0$ are acyclic~\cite[Section VII.4]{brown}, its $E^1$-page has non-trivial groups~$C_s(X)$ concentrated in the first column. The differential is induced from $\dd^0_v$. Hence, the second page consists of the homology groups $\hh_*(X)$. Therefore, this spectral sequence collapses at the second page, yielding  
\begin{equation}\label{eq:firstisossp}
\hh_*(\Tot X)\cong \hh_*(X) \ .
\end{equation}

The second spectral sequence $IIE$ is called the \emph{Mayer-Vietoris spectral sequence}. From now on, we will simply write $E$ instead of $IIE$. 
\begin{defn}\label{def:1pageMV}
The  first page of the Mayer-Vietoris spectral sequence associated to a simplicial complex $X$ and cover $\cU$ 
is  given by 
\begin{equation}\label{eq:E1}
E^1_{p,q}=\bigoplus_{\sigma\in \mathrm{N}_p(\cU)} \hh_q(U_{\sigma})
\end{equation}
with differential $\dd^{(1)}\colon E^1_{p,q}\to E^1_{p-1,q}$  induced by $\dd^0_h$. 
\end{defn}

\begin{rem}
The group $E^1_{p,q}$  in Equation~\eqref{eq:E1} coincides with the chain group $C_p(\Nerve(\cU);\cH_q)$, where $\cH_q$ is the coefficient system described in Example~\ref{ex:coeffsys}.
\end{rem}

Then, the $E^2$-page of $IIE$ is given by
\[
E^2_{p,q}=\hh_p(E^1_{*,q})=\hh_p(\Nerve(\cU);\cH_q) \ .
\]
As previously remarked, this spectral sequence converges to the homology of the total complex $\Tot X$. Therefore, we get convergence of the Mayer-Vietoris spectral sequence to the homology of $X$ by Equation~\eqref{eq:firstisossp}. 

\begin{rem}
Assume that the elements of the cover of $X$ have homology $\hh_i(U_\sigma)=0$ for all $\sigma$ and $i\geq k$.
Then, the differential $\dd^{i}$ on the $i$-th page must be trivial for $i\geq k+2$. Thus, in such a case, the Mayer-Vietoris spectral sequence converges at the page $E^{k+2}$.
\end{rem}

As a consequence, if each $U_\sigma$ is acyclic, the described spectral sequence collapses at the second page (furthermore, it has non-zero groups only at $q=0$), and we recover the classical nerve lemma -- \emph{cf}.~\cite[Theorem VII (4.4)]{brown}:

\begin{thm}[Nerve Lemma]\label{thm:nervethm}
Let $X$ be a finite simplicial complex, $\cU$ a cover by subcomplexes, and suppose that every non-empty intersection $U_\sigma$ is acyclic. Then, $\hh_*(X)\cong \hh_*(\Nerve(\cU))$.
\end{thm}

When the subcomplexes $U_\sigma$ are not acyclic, the conclusion of the Nerve Lemma does not hold. Nonetheless, the Mayer-Vietoris spectral sequence eventually converges to the homology of $X$.

\begin{rem}\label{rem:cech}
The results outlined in this section are a special case of a more general construction. Let $\cF$ be a sheaf on a topological space $X$.
If $\cU$ is a cover of $X$ which is $\cF$-acyclic (\emph{i.e.}~$\cF$ is acyclic on the finite intersections of $\cU$), then the \v{C}ech cohomology of $\cU$ with coefficients in $\cF$ coincides with the sheaf cohomology of $X$. Furthermore, if $\cU$ consists of two open subsets of $X$, we recover the classical Mayer-Vietoris sequence for the sheaf $\cF$.
\end{rem}

\subsection{The {anti-star} cover}\label{sec:antistar}

For our purposes, it is particularly interesting to consider a special type of covers of simplicial complexes. These covers are obtained from complements of vertex stars, and are commonly known as anti-star covers. Let $X$ be a simplicial complex which is not the standard simplex $\Delta^m$.
 
\begin{defn}\label{defn:astarcomplex}\label{defn:covering}
For each vertex $v$ in $X$, denote by $\mathrm{ast}_X(v)$ the subcomplex of $X$ spanned by the vertices in $V(X)\setminus\{v\}$. The associated cover $\cU^{\mathrm{ast}}_X=\{\mathrm{ast}_X(v)\}_{v\in V(X)}$ is called the \emph{anti-star} cover of $X$.
\end{defn}

Equivalently, each $\mathrm{ast}_X(v)$  is obtained from  $X $ by removing the open star of $v$. When clear from the context, we will drop the dependency on $X$ and simply write $\mathrm{ast}(v)$ and $\cU^{\mathrm{ast}}$. 
Anti-star subcomplexes contain homotopical information about the simplicial complex $X$. Indeed, if the inclusions in the (anti-)star is null-homotopic, if $\mathrm{ast}_X(v)$ is homotopy equivalent to a wedge of $n$-dimensional spheres,  and the link $\mathrm{lk}_X(v)$ is homotopy equivalent to a wedge of $(n-1)$-dimensional spheres, then $X$ is homotopy equivalent to a wedge of $n$-dimensional spheres~\cite[Lemma~5]{Omega}. Furthermore, if $v$ is a (non-isolated) vertex of $X$, there is a Mayer-Vietoris long exact sequence  
\[
\cdots \to \widetilde{\hh}_{i+1} (X)\to \widetilde{\hh}_{i}(\mathrm{lk}_X(v))\to \widetilde{\hh}_{i}(\mathrm{ast}_X(v))\to \widetilde{\hh}_{i}(X)\to \cdots 
\]
relating the (reduced) homology of $X$, of the link of $v$ and of the associated anti-star complex. When multiple vertices are considered at once, long exact sequences are not sufficient to determine the homology of $X$, and the Mayer-Vietoris spectral sequence comes into play.
\begin{rem}
The nerve associated to the anti-star cover can be easily seen to coincide with the standard simplex {$\Delta^m$}.
\end{rem}

In order to showcase some of the techniques used in Section~\ref{sec:applications}, some  sample computations of the Mayer-Vietoris spectral sequence associated to the anti-star cover are provided below.

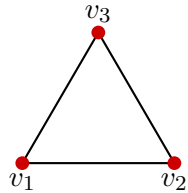
\begin{figure}[h]
\centering
	\begin{tikzpicture}[baseline=(current bounding box.center),line join = round, line cap = round]
		\tikzstyle{point}=[circle,thick,draw=black,fill=black,inner sep=0pt,minimum width=2pt,minimum height=2pt]
		\tikzstyle{arc}=[shorten >= 8pt,shorten <= 8pt,->, thick]
		\def\c{8}\def\d{.5}
		
		\node (e1) at (0,0) {};
		\node (e2) at (2,0) {};
		\node (e3) at (60:2) {};
		\draw[
  thick] (e1.center) -- (e2.center) -- (e3.center) -- (e1.center);
		\node[circle,fill=bunired,scale=0.5] at (e1) {};\node[below] at (e1) {$v_1$};
        \node[circle,fill=bunired,scale=0.5] at (e2) {};\node[below] at (e2) {$v_2$};
        \node[circle,fill=bunired,scale=0.5] at (e3) {};\node[above] at (e3) {$v_3$};
\end{tikzpicture}
	\caption{The simplicial complex from Example~\ref{ex:cycle}.}
	\label{fig:triangle}
\end{figure}

    \begin{figure}
	\begin{tikzpicture}
        \def\y{1}\def\x{3}
        \tikzstyle{sp}=[circle,fill=bunired,scale=0.3]
        \tikzstyle{sptr}=[circle,fill=black,scale=0.15]
        \tikzstyle{se}=[thick, bunired]
        \tikzstyle{setr}=[ black]

        \node (p123) at (0*\x,3*\y){\begin{tikzpicture}[scale=0.3]
        \node[sptr] (1) at (0,0){};
        \node[sptr] (2) at (2,0){};
        \node[sptr] (3) at (60:2){};
        \draw[setr] (1) -- (2);
        \draw[setr] (3) -- (2);
        \draw[setr] (1) -- (3);
        \end{tikzpicture}};
        
        \node (p12) at (-1*\x,2*\y){\begin{tikzpicture}[scale=0.3]
        \node[sp] (1) at (0,0){};
        \node[sptr] (2) at (2,0){};
        \node[sptr] (3) at (60:2){};
        \draw[setr] (1) -- (2);
        \draw[setr] (3) -- (2);
        \draw[setr] (1) -- (3);        \end{tikzpicture}};
        
        \node (p13) at (0*\x,2*\y){\begin{tikzpicture}[scale=0.3]
        \node[sptr] (1) at (0,0){};
        \node[sp] (2) at (2,0){};
        \node[sptr] (3) at (60:2){};
        \draw[setr] (1) -- (2);
        \draw[setr] (3) -- (2);
        \draw[setr] (1) -- (3);        \end{tikzpicture}};
        
        \node (p23) at (1*\x,2*\y){\begin{tikzpicture}[scale=0.3]
        \node[sptr] (1) at (0,0){};
        \node[sptr] (2) at (2,0){};
        \node[sp] (3) at (60:2){};
        \draw[setr] (1) -- (2);
        \draw[setr] (3) -- (2);
        \draw[setr] (1) -- (3);        \end{tikzpicture}};
        
        \node (p1) at (-1*\x,0.7*\y){\begin{tikzpicture}[scale=0.3]
        \node[sp] (1) at (0,0){};
        \node[sp] (2) at (2,0){};
        \node[sptr] (3) at (60:2){};
        \draw[se] (1) -- (2);
        \draw[setr] (1) -- (3);
        \draw[setr] (2) -- (3);
        \end{tikzpicture}};
        
        \node (p2) at (0*\x,1*\y){\begin{tikzpicture}[scale=0.3]
                \node[sp] (1) at (0,0){};
        \node[sptr] (2) at (2,0){};
        \node[sp] (3) at (60:2){};
        \draw[setr] (1) -- (2);
        \draw[setr] (2) -- (3);
                \draw[se] (1) -- (3);

        \end{tikzpicture}};
        
        \node (p3) at (1*\x,0.7*\y){\begin{tikzpicture}[scale=0.3]
        \node[sptr] (1) at (0,0){};
        \node[sp] (2) at (2,0){};
        \node[sp] (3) at (60:2){};
        \draw[setr] (1) -- (3);
        \draw[se] (2) -- (3);
        \draw[setr] (1) -- (2);

        \end{tikzpicture}};
        
        \node (p0) at (0*\x,-0.2*\y){\begin{tikzpicture}[scale=0.3]        \node[sp] (1) at (0,0){};
        \node[sp] (2) at (2,0){};
        \node[sp] (3) at (60:2){};
        \draw[se] (1) -- (3);
        \draw[se] (2) -- (3);
        \draw[se] (2) -- (1);
        \end{tikzpicture}};
        
        \draw[thick, dotted] (p123) -- (p12) -- (p1) -- (p0);
        \draw[thick,dotted] (p123) -- (p23) -- (p2) -- (p0);
        \draw[thick,dotted] (p123) -- (p13) -- (p3) -- (p0);
        \draw[thick,dotted] (p12) -- (p2);
        \draw[thick,dotted] (p23) -- (p3);
        \draw[thick,dotted] (p13) -- (p1);
	\end{tikzpicture}

\caption{The Boolean diagram associated to the nerve of the anti-star cover for $\partial \Delta^2$, with the empty and complete intersections added (top and bottom elements, respectively); red simplices denote the elements $U_i=\mathrm{ast}_X(v_i)$  and their intersections (\emph{cf.}~with Figure~\ref{fig:triang}).}
\label{fig:nvjvjh}
\end{figure}
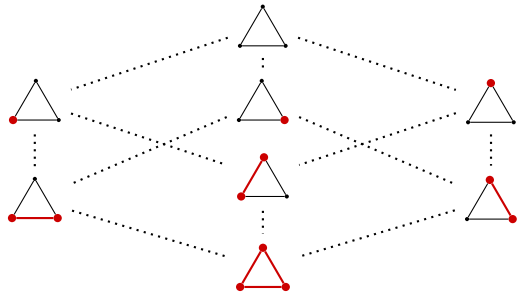

\begin{example}\label{ex:cycle}
Let $X = \partial \Delta^2$, as shown in Figure~\ref{fig:triangle}. Let $U_i=\mathrm{ast}(v_i)$ be the anti-star subcomplexes, so~$U_i=[v_{i+1},v_{i+2}]$, with indices modulo $3$. The intersection $U_i\cap U_{i+1}$ is given by the vertex $v_{i+2}$. Adding to $\Uast$ the empty and the complete intersections as well, produces  the Boolean poset
represented in Figure~\ref{fig:nvjvjh}. Applying the homology functor $\hh_q$ to each element of the poset, yields instead the decorated cube in Figure~\ref{fig:cuboE1}.  The directions of the edges of the cube follow the inclusions; in turn, the $p$-differentials, are directed from the $p$-simplices of the associated nerve to the $(p-1)$-simplices. 

In order to get the double complex $E^0_{p,q}$, for each $q$ take the module generated by the $q$-chains, and then sum them up:
\[
E^0_{2,q}=C_q(\emptyset),\quad E^0_{1,q}=\bigoplus_{i=1}^3 C_q(\{v_i\}), \quad E^0_{0,q}=\bigoplus_{i=1}^3 C_q(\{[v_i,v_{i+1}]\}) \ .
 \]
To turn to the first page,  take the homology in the $q$-direction; for each $q$, this results in  
the row shown at the bottom of Figure~\ref{fig:cuboE1}. 
 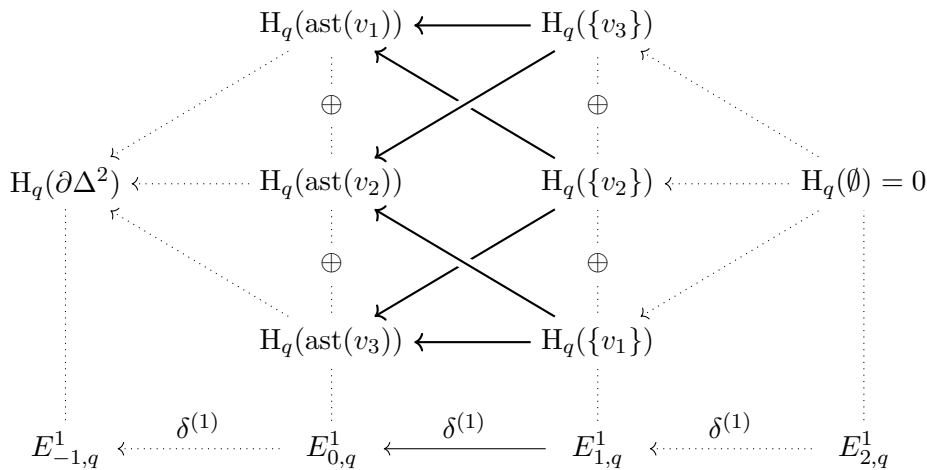
\begin{figure}[ht]
\begin{center}
\begin{tikzpicture}[scale = .7]

\draw[dotted] (0,0) -- (0,-4.5) node[below] {$E^1_{-1,q}$};

\draw[dotted] (5,3) -- (5,-4.5) node[below] {$E^1_{0,q}$};

\draw[dotted] (10,3) -- (10,-4.5) node[below] {$E^1_{1,q}$};

\draw[dotted] (15,0) -- (15,-4.5) node[below] {$E^1_{2,q}$};

\node[below] (uc0) at (0,-4.5) {\phantom{$\UC^0(X)$}};
\node[below] (uc1) at (5,-4.5) {\phantom{$\UC^0(X)$}};
\node[below] (uc2) at (10,-4.5) {\phantom{$\UC^0(X)$}};
\node[below] (uc3) at (15,-4.5) {\phantom{$\UC^0(X)$}};

\draw[dotted, ->] (uc1) -- (uc0) node[midway, above] {$\dd^{(1)}$};
\draw[->] (uc2) -- (uc1) node[midway, above] {$\dd^{(1)}$};
\draw[dotted, ->] (uc3) -- (uc2) node[midway, above] {$\dd^{(1)}$};

\node[fill, white] at (5,1.5){$\oplus$};
\node[fill, white] at (5,-1.5){$\oplus$};
\node  at (5,1.5){$\oplus$};
\node  at (5,-1.5){$\oplus$};

\node[fill, white] at (10,1.5){$\oplus$};
\node[fill, white] at (10,-1.5){$\oplus$};
\node  at (10,1.5){$\oplus$};
\node  at (10,-1.5){$\oplus$};

\node[fill, white] at (0,0) {$\hh_q(X)$};

\node[fill, white] at (5,3){${\Hh(X,(1,0,0))}$};
\node[fill, white] at (5,0){${\Hh(X,(0,1,0))}$};
\node[fill, white] at (5,-3){${\Hh(X,(0,0,1))}$};

\node[fill, white] at (10,3) {${\Hh(X,(1,1,0))}$};
\node[fill, white] at (10,0) {${\Hh(X,(1,0,1))}$};
\node[fill, white] at (10,-3) {${\Hh(X,(0,1,1))}$};

\node[fill, white] at (15,0) {${\Hh(X,(1,1,1))}$};

\node (a) at (0,0) {${\hh_q(\partial \Delta^2)}$};

\node (b1) at (5,3) {$\hh_q(\mathrm{ast}(v_1))$};
\node (b2) at (5,0) {$\hh_q(\mathrm{ast}(v_2))$};
\node (b3) at (5,-3){$\hh_q(\mathrm{ast}(v_3))$};

\node (c1) at (10,3) {$\hh_q(\{v_3\})$};
\node (c2) at (10,0) {$\hh_q(\{v_2\})$};
\node (c3) at (10,-3) {$\hh_q(\{v_1\})$};

\node (d) at (15,0) {$\hh_q(\emptyset)=0$};

\draw[dotted, <-] (a) -- (b1) node[midway,above,rotate =31] {}; 
\draw[dotted, <-] (a) -- (b2) node[midway,above] {}; 
\draw[dotted, <-] (a) -- (b3) node[midway,below,rotate =-29] {}; 

\draw[thick, <-] (b1) -- (c1) node[midway,above] {}; 
\draw[thick, <-] (b1) -- (c2) node[midway,above left,rotate =-29] {}; 

\draw[thick, <-] (b3) -- (c2) node[midway,below left,rotate =31] {}; 
\draw[thick, <-] (b3) -- (c3) node[midway,below] {}; 

\draw[line width = 5, white] (b2) -- (c1) ; 
\draw[line width = 5, white] (b2) -- (c3) ; 
\draw[thick, <-] (b2) -- (c1) node[midway,below left,rotate =31] {} ; 
\draw[thick, <-] (b2) -- (c3) node[midway,above left,rotate =-29] {}; 

\draw[dotted, <-] (c1) -- (d) node[midway,above,rotate =-29] {}; 
\draw[dotted, <-] (c2) -- (d) node[midway,above] {}; 
\draw[dotted, <-] (c3) -- (d) node[midway,below,rotate =31] {}; 
\end{tikzpicture}

\caption{The coefficient system $\cH_q$ on the nerve of $\partial \Delta^2$, augmented by adding the values on the empty and complete intersections. The direct sum, columnwise, yields the $q$-th row of $E^1$ in the Mayer-Vietoris spectral sequence. }\label{fig:cuboE1}
\end{center}
\end{figure}
The only non-trivial row is at $q=0$. 
The unique differential has a non-trivial kernel of rank~$1$. 
Taking the homology again, yields non-trivial homology groups~$E^2_{0,0}, E^2_{1,0}$, both isomorphic to $\Z$; all other groups are zero. Concluding, the Mayer-Vietoris spectral sequence converges at the second page (as all differentials are zero), yielding non-trivial classes in homological dimension $0$ and $1$; this corresponds to the fact that~$X$ is homotopic to~$S^1$.
\end{example}

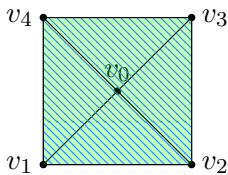
\begin{figure}[ht]
	\begin{tikzpicture}[scale=1.30]
    \tikzstyle{arc}=[shorten >= 3pt,shorten <= 3pt,->, thick]

	\tikzstyle{point}=[circle,thick,draw=black,fill=black,inner sep=0pt,minimum width=2pt,minimum height=2pt]
	\node[point] at (4,0) {};	\node[left] at (4,0) {$v_1$};
 \node[point] at (4.75,.75) {};
 \node[above] at (4.75,.75) {$v_0$};
	\node[point] at (5.5,0) {};
 	\node[right] at (5.5,0) {$v_2$};
	\node[point] at (4,1.5) {};	
 	\node[left] at (4,1.5) {$v_4$};	

	\node[point] at (5.5,1.5) {};
 	\node[right] at (5.5,1.5) {$v_3$};

	\draw[fill=green,opacity=0.25] (4,0) -- (5.5,0) -- (5.5,1.5) -- (4,1.5) --cycle; 
	\draw (4,0)  -- (5.5,0);		
	\draw (4,0)  -- (4,1.5);
	\draw (5.5,0)  -- (5.5,1.5);
	\draw (4,1.5)  -- (5.5,1.5);
	\draw (4,0)  -- (5.5,1.5);
 \draw (4,1.5)  -- (5.5,0);	
	\draw[pattern=north west lines, pattern color=bluedefrance]
 (4,0) -- (5.5,0) -- (5.5,1.5) -- (4,1.5) --cycle; 
	\end{tikzpicture}
 \caption{The cone over a loop of length four.}\label{fig:squarediags}
\end{figure}

\begin{example}\label{ex:sqdiags}
Consider the {contractible} simplicial complex $X$ shown in Figure~\ref{fig:squarediags}, \emph{i.e.}~the cone over a loop of length $4$. 
Then, as the Mayer-Vietoris spectral sequence converges to the homology of the disc, and all the subcomplexes $\mathrm{ast}(v_i)$, but $\mathrm{ast}(v_0)$, are contractible, the first and second page of the Mayer-Vietoris spectral sequence are the following:
\begin{equation*}
    \begin{tikzpicture}
  \matrix (m) [matrix of math nodes,
    nodes in empty cells,nodes={minimum width=5ex,
    minimum height=5ex,outer sep=-5pt},
    column sep=1ex,row sep=1ex]{
          q      &      &     &   &  & E^1\\
          1     &  \Z &  0  & 0 & 0 &\\
          0     &  \Z^5& \Z^{10} &  \Z^{12}  & \Z^{5}&\\
    \quad\strut &   0 &  1  &  2  & 3 &\strut p\\};
\draw[thick] (m-1-1.east) -- (m-4-1.east) ;
\draw[thick] (m-4-1.north) -- (m-4-6.north) ;
\end{tikzpicture}
\quad
    \begin{tikzpicture}
  \matrix (m) [matrix of math nodes,
    nodes in empty cells,nodes={minimum width=5ex,
    minimum height=5ex,outer sep=-5pt},
    column sep=1ex,row sep=1ex]{
          q      &      &     &   &  & E^2\\
          1     &  \Z &  0  & 0 & 0 &\\
          0     &  \Z & 0 &  \Z  & 0 &\\
    \quad\strut &   0 &  1  &  2  & 3 &\strut p\\};
\draw[thick] (m-1-1.east) -- (m-4-1.east) ;
\draw[thick] (m-4-1.north) -- (m-4-6.north) ;
\end{tikzpicture}
\end{equation*}
The differential $\dd^{(2)}\colon E^2_{2,0}\cong \Z\to \Z\cong E^2_{0,1}$ is non-trivial, and it kills the $\Z$-class at $E^2_{0,1}$. The third page contains only the homology of the point.
\end{example}

\begin{example}\label{ex:spheres}
Let $\Delta^{n+1}$ be the standard $(n+1)$-simplex, considered with its standard triangulation. Let $S^n$ be the sphere obtained by removing from $\Delta^{n+1}$ its interior. Then, the associated anti-star cover consists of contractible subcomplexes. Hence, the Mayer-Vietoris spectral sequence converges at the second page. As a consequence, $E^2$ contains a rank one component in degree $(0,0)$, and one  in degree $(n,0)$.
\end{example}

\section{The overlap between Mayer-Vietoris and \"uberhomology}

The aim of this section is to prove Theorem~\ref{thm:uber=MV}; or, more explicitly, to provide the identification between the second page of the Mayer-Vietoris spectral sequence associated to the \emph{anti-star cover} and the $0$-degree \"uberhomology. To this end, the first step is to extend the computational framework of the Mayer-Vietoris spectral sequence to include the empty intersection of the elements of the cover as well.

Let $X$ be a simplicial complex, and let $\Uast$ be its
associated anti-star cover. Assume that $X$ is not the standard simplex. By Equation~\eqref{eq:E1}, the first page of the associated Mayer-Vietoris spectral sequence is  $$E^1_{p,q}=\bigoplus_{\sigma\in \mathrm{N}_p(\cU^{ast})} \hh_q(U_{\sigma}) \ .$$  In order to include the empty intersection $U_\emptyset\coloneqq X$, it is possible to augment both the double chain complex $E^0_{p,q}$ and the first page of the spectral sequence in degree $p=-1$ with the homology of $X$, 
by setting
\[
E^1_{-1,q}\coloneqq \hh_q(X) \ .
\] 
The horizontal differential of the first page, induced by inclusions, naturally extends to the $(-1)$-column as well.

\begin{defn}
The \emph{augmented Mayer-Vietoris spectral sequence} of $X$ relative to a cover $\cU$ is the spectral sequence with $E^1$-page given by $E^1_{p,q}=\bigoplus_{\sigma\in \mathrm{N}_p(\cU^{ast})} \hh_q(U_{\sigma})$, augmented in degree~$-1$ with $E^1_{-1,q}\coloneqq \hh_q(X)$, and differentials induced by inclusions. \end{defn}

\begin{rem}
Consider the double chain complex  $E^0_{p,q}$, augmented in degree $p=-1$ with   $C_q(X)$; \emph{i.e.}~set $E^0_{-1,q}\coloneqq C_q(X)$.   As the rows of such augmented complex are exact~\cite[Section VII.4]{brown},  
by the Acyclic Assembly Lemma~\cite[Lemma~2.7.3]{Weibel}, the total complex associated to~$E^0_{*,*}$ is acyclic. Hence, the augmented Mayer-Vietoris spectral sequence associated to (the second filtration of) $E^0_{*,*}$ converges to an acyclic complex. 
\end{rem}

We provide some examples,  extending those already discussed in Section~\ref{sec:antistar}.

\begin{example}\label{ex:cycle2}
In parallel with Example~\ref{ex:cycle},  consider the spectral sequence obtained by augmenting the first page in degree $-1$ with the homology of $X$.  The first and second page now become the following:
\begin{equation*}
\begin{tikzpicture}
  \matrix (m) [matrix of math nodes,
    nodes in empty cells,nodes={minimum width=5ex,
    minimum height=5ex,outer sep=-5pt},
    column sep=1ex,row sep=1ex]{
          q      &      &     &     & E^1\\
          1     &  \Z &  0  & 0 & \\
          0     &  \Z  & \Z^3 &  \Z^3  & \\
    \quad\strut &   -1  &  0  &  1  & \strut p\\};
\draw[thick] (m-1-1.east) -- (m-4-1.east) ;
\draw[thick] (m-4-1.north) -- (m-4-5.north) ;
\end{tikzpicture}
\quad
\begin{tikzpicture}
  \matrix (m) [matrix of math nodes,
    nodes in empty cells,nodes={minimum width=5ex,
    minimum height=5ex,outer sep=-5pt},
    column sep=1ex,row sep=1ex]{
          q      &      &     &     & E^2\\
          1     & \node(b) {\Z}; &  0  & 0 & \\
          0     &  0  & 0 & \node(a) {\Z};  & \\
    \quad\strut &   -1  &  0  &  1  & \strut p\\};
\draw[thick] (m-1-1.east) -- (m-4-1.east) ;
\draw[thick] (m-4-1.north) -- (m-4-5.north) ;
\end{tikzpicture}
\end{equation*}
The unique differential $\dd^{(2)}\colon E^2_{1,0}\to E^2_{-1,1}$ is an isomorphism; hence, the third page is trivial. 
Analogously, for the square of Example~\ref{ex:sqdiags}, the second page is completely trivial, except for the bidegrees $(0,1)$ and $(2,0)$, where it is $\Z$.
\end{example}

\begin{example}
Consider the spheres from Example~\ref{ex:spheres}.  It is easy to see that the second page of the augmented Mayer-Vietoris spectral sequence is completely trivial, except for in degrees $(-1, n)$ and $(n, 0)$. The augmented spectral sequence collapses at page $n+1$. \end{example}

We can now proceed with the proof of Theorem~\ref{thm:uber=MV};

\begin{proof}[Proof of Theorem~\ref{thm:uber=MV}]
For a given set of vertices $\{v_{i_1},\dots, v_{i_k}\} \subseteq V(X)$,  the induced subcomplex $$X \langle v_{i_1},\dots, v_{i_k}\rangle \subseteq X$$ they span can be regarded as the ``$1$-coloured'' component of the complex $(X,\e)$, $\e$ being the bicolouring on $X$ assigning $1$ to each $v_{i_j}$ and $0$ otherwise. 
Equivalently, $X \langle v_{i_1},\dots, v_{i_k}\rangle$  is obtained by intersecting the anti-star subcomplexes $U_s$ for all $s\notin \{i_1, \dots, i_k\}$.
    
Now observe that for a fixed $q\geq 0$, by the definition of the chains in Equation~\eqref{eq:uberchain},  restricted to the $0$-degree, 
we obtain an identification of the groups  $E^1_{p,q}$ with $\bigoplus_{\ell(\e)=m-p - 1} \hh_q(X,\e)$; this is the \"uberhomological degree $m-p -1$ component of the $0$-degree of the \"ubercomplex. When restricting to the base field $R=\Z_2$, the  $p$-differential coincides with the differential in Equation~\eqref{eq:uberdiff}. Furthermore, the agreement of the differentials extends  to a general ring of coefficients~$R$ after choosing a sign assignment on the appropriate Boolean poset~\cite{primo}; the agreement does not depend on the choice of the sign assignment by Proposition~\ref{rem:ubergencoeff} and \cite[Theorem~3.16 \& Corollary~3.18]{primo}. 
By Definition~\ref{def:uber}, the homology of this chain complex is the \"uberhomology of $X$. On the other hand, it yields the second page of the augmented Mayer-Vietoris spectral sequence. This gives the complete identification~$E^2_{p,q}\cong \BH^{m-p-1}_q(X)$ for $p\geq -1$ and $q\geq 0$.
\end{proof}

In other words, the $0$-degree \"uberhomology of $X$ coincides with the homology of the nerve of the anti-star cover, with coefficients in the functor $\cH_*$  defined in Example~\ref{ex:coeffsys}.
Observe that the definition of anti-star cover can be extended verbatim to regular CW-complexes. 
Furthermore, Theorem~\ref{thm:uber=MV} allows us to extend also the definition of \"uberhomology to regular CW-complexes. This will be used in Example~\ref{ex:polyngb} and Corollary~\ref{cor:chordalchar}.

\begin{cor}\label{cor:diff}
The \"uberhomology groups $\BH^j_i (X)$
inherit a further differential 
\[
\dd^{(2)}\colon \BH^j_i (X)\to \BH^{j+2}_{i+1} (X) 
\]
for all $i\geq 0$ and $0\leq j\leq m=V(X)$. Hence, $(\BH^j_i (X), \dd^{(2)})$ is a chain complex.
\end{cor}

 \begin{proof}
The differential $\dd^{(2)}\colon \BH^j_i (X)\to \BH^{j+2}_{i+1} (X) $ is precisely the differential $\dd^{(2)}$ of the second page~$E^2_{*,*}$ of the augmented Mayer-Vietoris spectral sequence.
 \end{proof}

Note that the induced differential~$\dd^{(2)}$ is related to  the connecting homomorphism in the Mayer-Vietoris  long exact sequence. Furthermore, the transgression of the augmented Mayer-Vietoris spectral sequence induces a (partially defined) map
\[
\tau\colon \mathbb{H}^{k}(X)=\BH_{0}^{k}(X)\longrightarrow \BH_{m-k-1}^{m}(X) \ ,
\]
where $m=|V(X)|$ and $k=0,\dots, m-2$. 

As a consequence of Theorem~\ref{thm:uber=MV}, the next computations of $0$-degree \"uberhomology groups follow;

\begin{example}
The homology groups $\BH^j_i (X)$ of the simplicial complex in Example~\ref{ex:cycle} are all zero, except for $\BH^3_1 (X)$ and $\BH^1_0 (X)$, both isomorphic to $\Z$. The two classes are paired by the differential $\dd^{(2)}\colon \BH^1_0 (X)\to \BH^{3}_{1} (X) $, and the bold homology class is paired with the fundamental class of $X$.  In this case, $\delta^{(2)}$ at $\BH^1_0 (X)$ agrees with the transgression~$\tau$, and is an isomorphism.   Consider the square from Example~\ref{ex:sqdiags}; it has non-trivial homology groups~$\BH^4_1 (X)$ and $\BH^2_0 (X)$; these are again paired by the differential (the transgression) $\dd^{(2)}$, which is an isomorphism. 
 \end{example}

 \begin{example}
The only non-trivial $0$-degree \"uberhomology groups of the standard spheres $\partial\Delta^m$ of Example~\ref{ex:spheres} are $\BH^{m+1}_{m-1} (\partial\Delta^m)$ and $\BH^1_0 (\partial\Delta^m)$. Note that, in such case, the bold homology class and the class of $\BH^{m+1}_{m-1} (\partial\Delta^m)$ are still paired, but by higher differentials in the augmented spectral sequence. However, the transgression map still yields an isomorphism between the groups~$\BH^{m+1}_{m-1} (\partial\Delta^m)$ and $\BH^1_0 (\partial\Delta^m)$.
\end{example}

 \begin{figure}[ht]
	\begin{tikzpicture}[scale=1.3]
	\tikzstyle{point}=[circle,thick,draw=black,fill=black,inner sep=0pt,minimum width=2pt,minimum height=2pt]
	\node[point] at (2,0) {};
     \node[below] at (2,0) {$v_0$};

    \node[point] at (1,1) {};
    \node[left] at (1,1) {$v_1$};

	\node[point] at (3,1) {};	
     \node[right] at (3,1) {$v_2$};

	\node[point] at (2,2) {};
	    \node[above] at (2,2) {$v_3$};

	\draw[fill=green,opacity=0.25] (2,0) -- (1,1) -- (2,2) -- (3,1) --cycle; 
	\draw (2,0)  -- (1,1);		
	\draw (1,1)  -- (2,2);
	\draw (2,2)  -- (3,1);
	\draw (3,1)  -- (2,0);
	\draw (1,1)  -- (3,1);
	\draw[pattern=north west lines, pattern color=bluedefrance]
(2,0) -- (1,1) -- (2,2) -- (3,1) --cycle; 
	\end{tikzpicture}
 \caption{Suspension of the $1$-simplex.}\label{fig:gluedtriangles}
\end{figure}
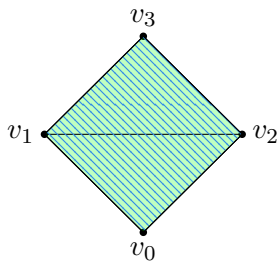

\begin{example}
Consider the simplicial complex $X$ obtained from two $2$-simplices, glued together along one edge, with vertices as in Figure~\ref{fig:gluedtriangles}. Call $v_0, v_3$ the external vertices and $v_1, v_2$ the vertices of the common edge. Then, the subcomplexes $\mathrm{ast}_X(v_i)$ and all the possible intersections, except for  $\mathrm{ast}_X(v_1)\cap \mathrm{ast}_X(v_2)$, are contractible. The subcomplex $\mathrm{ast}_X(v_1)\cap \mathrm{ast}_X(v_2) = \{ v_0, v_3\}$ is disconnected. 
The spectral sequence converges at the second page, where it is completely trivial. Hence, the homology groups $\BH^j_i (X)$ of $X$ are all zero.  
\end{example}

\section{Applications}\label{sec:applications}

In this section, we provide some consequences and applications of Theorem~\ref{thm:uber=MV}. First, we recall the definition of $d$-Leray complexes. 

\begin{defn}\label{def:leray}
A CW-complex $X$ is \emph{$d$-Leray} if the reduced homology of all induced subcomplexes of $X$ is trivial for all $i\geq d$. A cover $\cU$ of $X$ is $d$-Leray if all its elements  are $d$-Leray.
\end{defn}

Equivalently, by \cite[Proposition~3.1]{kalai}, a simplicial complex~$X$ is $d$-Leray if $\widetilde{\hh}_i(\mathrm{lk}_X(\sigma))=0$ for all simplices~$\sigma$ in $X$ and $i\geq d$. The property of being $d$-Leray has consequences on the convergence of the Mayer-Vietoris spectral sequence. For example, the following result can  be readily deduced from Theorem~\ref{thm:uber=MV}:

\begin{prop}\label{prop:gentrees}
Let $X\neq \Delta^m$ be a connected, contractible, simplicial complex. If the anti-star cover of $X$ is $1$-Leray, then all the homology groups $\BH^j_i (X)$ are zero.
\end{prop}

\begin{proof}
If the anti-star cover is $1$-Leray, then the first page of the augmented Mayer-Vietoris spectral sequence has non-trivial elements only for $q=0$. The spectral sequence converges to zero at the second page, since $X$ is contractible. 
\end{proof}

This proposition provides a generalization of~\cite[Theorem~5.1]{domination}: the bold homology of a connected tree with at least $3$ vertices is trivial. Indeed, in this case, the anti-star cover is $1$-Leray, and the bold homology is the $0$-part of $\BH^*_* (X)$. Another example is given by ``polygonal neighbourhoods of trees'':

\begin{example}\label{ex:polyngb}
 Let $\tT$ be a tree and choose an abstract full polygon $P$ (\emph{i.e.}~a $2$-cell) on at least~$n$ edges, where $n=\max_{v\in \tT} \deg(v)+2$. Then, for each vertex $v$ of $\tT$ take a copy $P^v$ of $P$ at~$v$; for each edge~$(v,w)$, glue $P^v$ and $P^w$ along a free face. 
As in $\tT$ there are no cycles, the construction yields a contractible CW-complex $X(\tT)$ (it is a simplicial complex only when $P$ is a triangle). Note that the construction may yield different complexes, even for a fixed choice of $P$. In any case, the anti-star cover of $X(\tT)$ is $1$-Leray.  Hence,  if $\tT\neq \Delta^1$, all the groups~$E^2_{*,*}(X(\tT))$ of the associated augmented Mayer-Vietoris spectral sequence are trivial -- by the same arguments of Proposition~\ref{prop:gentrees}. Consider now the $1$-skeleton of $X(\tT)$, which is a connected graph, and   has homology in dimension $1$ of rank $|V(\tT)|$. The anti-star cover associated to such graph is not $1$-Leray any more. However, observe that the $(q=0)$-row of the augmented Mayer-Vietoris spectral sequence is still trivial (whether the polygons are boundary of $2$-cells does not affect the number of connected components, nor the induced $p$-differentials). Therefore, the bold homology of polygonal neighbourhoods of trees ($\neq \Delta^1$)   is completely trivial.
\end{example}

Denote by $\tI_m$  the path graph on $m$ vertices, and by $\square$ the Cartesian product of graphs.

\begin{cor}\label{cor:grids}
If $m\geq 3$, then all graph products $\tI_m\square \tI_2$ have trivial bold homology.
\end{cor}

\begin{proof}
Note that the product $\tI_m\square \tI_2$ can be seen as a polygonal neighbourhood of $\tI_m$, by choosing the polygon~$P$ to be a square. The product $\tI_m\square \tI_2$ is not contractible, nor the anti-star cover is $1$-Leray. However, the same  reasoning as in Example~\ref{ex:polyngb} shows that the graph $\tI_m\square \tI_2$ has trivial bold homology. 
\end{proof}

As a consequence, the connected domination polynomial of $\tI_m\square \tI_2$, evaluated at $-1$, is trivial.   We point out that, in general, counting connected dominating sets for grids is not straightforward; see also~\cite{Srinivasan, goto2021connected}.

A priori, one can use properties of the anti-star cover of closed manifolds to deduce properties of their $1$-skeleton's bold homology.
For instance, assume that $M$ is a closed connected  non-orientable $n$-manifold.  Consider a $m$-vertex triangulation $T$ of $M$. If the anti-star cover is $1$-Leray, then the first page of the augmented Mayer-Vietoris spectral sequence $E^1_{p,q}$ is zero for all $p\geq 0$ and $q\geq 1$, whereas  $E^{1}_{-1,n-1} = H_{n-1}(M;\Z) \cong \Z_2$, see~\cite[Corollary~3.28]{hatcher}. The spectral sequence converges  to an acyclic total complex, yielding an isomorphism:
\[
\dd^{(n)}\colon E^n_{n-1,0}\longrightarrow E^n_{-1,n-1}\cong \Z_2 \ .
\]
Since $E^2_{n-1,0}\cong E^n_{n-1,0}$, then $\BH^{m}_{n-1}(M;\Z)\cong \mathbb{H}^{m-n}(M;\Z)\cong \Z_2$, and the bold homology would contain torsion. Unfortunately, the anti-star cover of a non-orientable closed connected $n$-manifold is $n$-Leray. Therefore, we can not infer the existence of torsion classes in bold homology. 

\begin{q}\label{question:torsion}
Is it possible to find torsion in the integral bold homology of a simplicial complex?
\end{q}

Given a simple graph $\tG$, denote by $\Fl(\tG)$ its flag complex, \emph{i.e.}~the simplicial complex whose simplices are given by the complete subgraphs of $\tG$. As remarked in \cite[Section~3.2]{MR3205172}, the only $0$-Leray complexes are the standard simplices~$\Delta^n$, whereas $1$-Leray complexes are flag complexes on chordal graphs.  An example of $2$-Leray and $3$-Leray complexes is given by flag complexes of line graphs of complete bipartite graphs, and complete graphs, respectively~\cite[Theorem~1.1]{HOLMSEN2022105618}. A class generalising bipartite graphs is given by triangle-free graphs.

\begin{prop}\label{prop:triangle_free}
Let $\tG$ be a connected triangle-free graph on $m$ vertices. Then, 
\begin{enumerate}[label= {\rm (\roman*)}]
\item $\mathbb{H}^{m}(\tG;\Z)=\mathbb{H}^{m-1}(\tG;\Z)=0$;
\item $\mathbb{H}^{m-2}(\tG;\Z)$ and $\BH^{m}_1(\tG;\Z)$ are quotients of $\hh_1(\tG;\Z)$;
\item $\mathbb{H}^{j}(\tG)\cong \BH^{j+2}_1(\tG)$ for all $0\leq j\leq m-2$.
\end{enumerate}
\end{prop}

\begin{proof}
As $\tG$ is triangle-free, it follows that $\Fl(\tG)=\tG$, and the anti-star cover of $\Fl(\tG)=\tG$ is $2$-Leray. An inspection of the spectral sequence, together with Corollary~\ref{cor:diff}, yields short sequences
\[
0\to \mathbb{H}^{j}(\tG) = \BH^j_0 (\tG) \overset{\dd^{(2)}}{\longrightarrow} \BH^{j+2}_{1} (\tG) \to 0
\]
Since the spectral sequence must converge at the third page, $\dd^{(2)}$ is an isomorphism. The statements follow by Theorem~\ref{thm:uber=MV}.
\end{proof}

In virtue of Lemma~\ref{lem:flaggraphs}, the above proposition implies that the $0$-degree \"uberhomology of flag complexes on triangle-free graphs is completely determined by their bold homology. \\

The effect of graph cones on bold homology of graphs was explored in~\cite[Proposition~5.3]{domination}. 
Denote by ${\rm Cone}(X)$ the cone of a simplicial complex~$X$.
Recall that, given two chain complexes $(C_*,\delta^C_*)$, and $(D_*,\delta^D_*)$, and a chain map $\psi: C_* \to D_*$ the \emph{mapping cone of $\psi$} is the chain complex defined as follows: 
\[ {\rm Cone}(\psi) = D_* \oplus C_{*-1},\quad \partial_{\rm Cone} = \begin{pmatrix} \delta_*^D & -\psi_* \\ 0 & \delta^C_{*-1}\end{pmatrix}\ .\]
 The following result is a partial generalisation of both \cite[Proposition~5.3]{domination} and~\cite[Proposition~7.11]{uberhomology}.

\begin{prop}\label{prop:cone}
Let $X\neq \Delta^{m-1}$  be a simplicial complex on $m$ vertices. Then, there is an isomorphism 
\[ \ddot{\rm B}^*_*({\rm Cone}(X)) \cong \ddot{\rm B}_*^{*}(X)\ \]
of bigraded modules. 
\end{prop}
\begin{proof}
Denote by $v_1,...,v_m$ the ordered vertices of $X$, and by $v_0$ the coning vertex in ${\rm Cone}(X)$.  
Set $U_i =  \mathrm{ast}_X(v_i)$, $U'_0 = \mathrm{ast}_{{\rm Cone}(X)}(v_0)$, and $U'_i =  \mathrm{ast}_{{\rm Cone}(X)}(v_i) \simeq {\rm Cone}(U_i)$. 
Note that all the~$U'_i$'s, as well as their intersections, are non-empty and contractible. Furthermore, $U'_0 = X$, and under this identification, each~$U_i$ corresponds to the intersection~$U'_0 \cap U'_i$. Therefore, we get the isomorphisms 
\[ E^{1}_{p,q}({\rm Cone}(X)) \cong E^1_{p-1,q}(X) \oplus \bigoplus_{i_1 < \dots < i_{p+1}} \hh_q(U'_{i_1}\cap \cdots \cap U'_{i_{p+1}}) = \begin{cases} E^1_{p-1,q}(X) \oplus \mathbb{Z}^{\binom{m}{p+1}} & \text{if }q =0\\ E^1_{p-1,q}(X) & \text{otherwise}\end{cases}\]
for all $p\geq -1$ and $q$, with the conventions that $E^1_{-1,q}(X)=\hh_q(X)$ and $E^1_{-2,q}(X)= 0$. 
We claim that:
\[ E^{1}_{p,q}({\rm Cone}(X)) = {\rm Cone}\left(\phi\colon  E^{1}_{p,q}(X)\to \widetilde{C}_p(\Delta^{m-1})   \right)\ ,\]
where $\phi$ is a graded chain map, $\Delta^{m-1} = [1,...,m]$, and $\widetilde{C}_p(\Delta^{m-1})$ is concentrated in $q$-degree~$0$.
To see this, identify the group $ \hh_0(U'_{i_1}\cap \cdots \cap U'_{i_{p+1}})$ with the summand in $\widetilde{C}_p(\Delta^{m-1})$ spanned by the simplex $[i_1,...,i_{p+1}]$. This identification provides the map $\phi$. 
Since $\widetilde{C}_*(\Delta^{m-1})$ is acyclic, the statement follows.
\end{proof}

The above reasoning can be extended to suspensions of simplicial complexes:
\begin{thm}\label{thm:susp}
Let $X$ be a connected simplicial complex on $m$ vertices, and denote by~$\Sigma (X)$ the suspension of $X$. Then, for $q\neq 1$ the following isomorphisms exist
\[ \ddot{\rm B}^j_{q}(\Sigma (X)) \cong \begin{cases} \ddot{\rm B}^{*}_{q}(X) \oplus \ddot{\rm B}^{*+2}_{q-1}(X) & \text{if }q>1,\\ \ddot{\rm B}^j_{0}(X) \oplus \Z_{(2)} & \text{if }q=0 \text{ and }X^{(1)}\neq \tK_{m},\\ 0 & \text{if }q=0 \text{ and }X^{(1)} = \tK_m,\\ \end{cases}\]
where $\Z_{(2)}$ indicates a copy of $\Z$ in \"uberhomological degree $2$.
\end{thm}

\begin{proof}
Let $v_1 ,..., v_m$ be the vertices of $X$ and write $\Sigma (X) = X \ast \{ p_1,p_2\}$.
For each $I\subseteq \{ 1,...,k\}$, possibly empty, we shall write $\mathcal{U}_I$ for the sub-complex of $X$ spanned by $\{ v_j \}_{j\notin I}$. Similarly, given $I\subseteq \{ 1,...,m\}$ and $J\subseteq \{1,2\}$, we denote by $\mathcal{U}_I^{J}$ the sub-complex of $\Sigma(X)$ spanned by~$\{ v_j \}_{j\notin I}\cup \{ p_r \}_{r\notin J}$.
Note that there are identifications $\mathcal{U}^{\emptyset}_{I} = \Sigma (\mathcal{U}_I)$, $\mathcal{U}^{\{ 1\}}_{I} = \mathcal{U}^{\{ 2\}}_{I} ={\rm Cone} (\mathcal{U}_I)$, and $\mathcal{U}^{\{ 1,2\}}_I = \mathcal{U}_I$.

By definition, the first page of the Mayer-Vietoris spectral sequence decomposes as
\begin{equation}
\label{eq:dec}
E_{p,q}^{1} = C^{0}_{p,q} \oplus C^{1}_{p-1,q}  \oplus C^{2}_{p-2,q}
\end{equation} 
where 
\[ C^{0}_{p,q} \coloneqq \bigoplus_{\vert I \vert - 1= p} \hh_q (\mathcal{U}_I^\emptyset), \; C^{1}_{p-1,q} \coloneqq \bigoplus_{\vert I \vert - 1 = p-1} \hh_q (\mathcal{U}_I^{\{ 1 \}})\oplus \hh_q (\mathcal{U}_I^{\{ 2 \}}),  \;  C^{2}_{p-2,q} \coloneqq  \bigoplus_{\vert I \vert - 1 = p-2} \hh_q (\mathcal{U}_I^{\{ 1,2 \}})\ .\]
For each $q$, the inclusions $\mathcal{U}_{I \setminus \{ i \}}^{J} \subset \mathcal{U}_{I}^{J}$ with $J\subseteq \{ 1,2\}$, induce differentials $\partial_0$, $ \partial_1$, and $\partial_2$ on $C^{0}_{*,q}$, $C^{1}_{*,q}$, and $C^{2}_{*,q}$, respectively, endowing them with the structure of chain complexes.
Similarly, the inclusions~$\mathcal{U}_{I}^{J} \subset \mathcal{U}_{I}^{J \setminus \{ j \}}$, induce chain maps $\phi_s \colon C^{s}_{*,q} \to C^{s-1}_{*,q}$, for $s=1,2$.
The signs of $\phi_1$ and $\phi_2$ can be chosen so that the differential on $E^1$ can be written, with respect to the decomposition in Equation \eqref{eq:dec}, as follows:
\[\delta^{(1)} = \begin{pmatrix} \partial_0 & -\phi_1 & 0 \\ 0 & \partial_1 & -\phi_2 \\ 0 & 0 & \partial_2 \end{pmatrix}\]
We can thus write $(E^1_{*,q}, \delta^{(1)})$ as an iterated cone:
\[E^1_{*,q} = {\rm Cone}((\phi_1,0)\colon {\rm Cone}(\phi_2: C^2_{*,q}\to C^1_{*,q} ) \to C^0_{*,q})\ . \]

Using the identifications $\mathcal{U}^{\emptyset}_{I} = \Sigma (\mathcal{U}_I)$, $\mathcal{U}^{\{ 1\}}_{I} = \mathcal{U}^{\{ 2\}}_{I} ={\rm Cone} (U_I)$, and $\mathcal{U}^{\{ 1,2\}}_I = \mathcal{U}_I$, we can obtain isomorphisms of chain complexes:
\[ C^0_{*,q} \cong \begin{cases} C_*(\Delta^{m-1})\oplus \Z_{(m- 1)} & \text{if }q=0\\  E^{1}_{*,q - 1}(X) & \text{if }q>1\\ \end{cases}\ , \quad C^1_{*,q} \cong \begin{cases} C_*(\Delta^{m-1})\oplus C_*(\Delta^{m-1}) & \text{if }q=0\\   0 & \text{if }q>1\\ \end{cases}\]
and
\[ C^2_{*,q} \cong E^{1}_{*-2,q}(X)\ .\]
Therefore, for $q=0$ we obtain strong deformation retracts (in the sense of \cite[Definition~4.3]{BN}) $C^0_{*,0} \simeq \Z_{(m-1)}$ and $C^1_{*,0} \simeq 0$. In turn (\emph{cf.}~\cite[Lemma~4.5]{BN}) these strong deformation retracts induce the following strong deformation retracts:
\[{\rm Cone}(\phi_2: C^2_{*,q}\to C^1_{*,q} ) \simeq C^2_{*-1,q}\cong E^{1}_{*-1,q}(X),\quad q\neq 1,\]
and thus
\[E^1_{*,q} \simeq \begin{cases} {\rm Cone}( \psi : E^{1}_{*-1,0}(X) \to \Z_{(m- 1)}) & \text{if }q=0\\   E^{1}_{*-2,q}(X) \oplus  E^{1}_{*,q - 1}(X) & \text{if }q>1\\ \end{cases}\]
for an appropriate chain map $\psi$.
This proves the statement for $q> 1$.
Now, if $X^{(1)} \neq  \tK_{m}$ then $\bH_1(X^{(1)}) = \BH_0^{1}(X)= 0$ (see \cite[Theorem~1.4]{domination}). 
In fact, we can prove that $E^1_{*,0}(X)$ (strongly) deformation retracts onto a complex $C_*$ supported in degrees strictly lower than $m$, \emph{cf.}~\cite[Alternative proof of Proposition~4.3]{domination}. This implies that
\[ E^1_{*,0} \simeq {\rm Cone}( 0 : C_{*-1} \to \Z_{(m- 1})) =  C_{*-2} \oplus \Z_{(m- 1)}\ .\]
Since $C_*$ is the chain complex whose homology is $\bH_{m-*}(X^{(1)}) = \ddot{\rm B}^{m-*}_{0}(X)$, the statement follows for~$q=0$ and $X^{(1)} \neq  \tK_{m}$.
Finally, consider the case $X^{(1)} =  \tK_{m}$.
By Lemma~\ref{lem:flaggraphs}, $\hh_{m+1}(E^{1}_{*,0}) = \bH_{m -* -2}((\Sigma(X))^{(1)})$ can be computed by considering any simplicial complex $Y$ such that $Y^{(1)} = (\Sigma(X))^{(1)}$.
For instance, we can take $Y$ to be $\Sigma \Delta^{m-1}$. Since $Y$ is contractible and its anti-star cover is $1$-Leray, it follows from Proposition~\ref{prop:gentrees} that $\bH_{*}((\Sigma(X))^{(1)}) = 0$.
\end{proof}

Observe that we cannot apply the same reasoning as in the proof of Theorem~\ref{thm:susp} for $q>1$ to obtain something about the case $q=1$. This is due to the fact that $\mathrm{H}_1(\Sigma X)\neq \mathrm{H}_{0}(X)$. We can see also that $\BH_{1}^{*}(\Sigma(X))$ is not necessarily trivial, by taking $X$ the linear graph with three vertices -- \emph{cf.}~Example~\ref{ex:sqdiags}. 

To conclude, we provide the proof of Theorem~\ref{cor:chordalchar2}:

\begin{proof}[Proof of Theorem~\ref{cor:chordalchar2}]
Let $X$ be a finite connected CW-complex and denote by $\tG$ its $1$-skeleton.
Since the anti-star cover of $X$ is $1$-Leray, 
the augmented Mayer-Vietoris spectral sequence converges   to the trivial group. Furthermore, the only non-trivial groups in the $E^2$-page are in bidegrees $(*,0)$ and $(-1,*)$. The \"uberhomology groups $\BH^j_i (X)$ are all zero, except for $\BH^{m-p-1}_0 (X)\cong \hh_p(X)=\BH_p^m (X)$. The isomorphism between the groups $\BH^{m-p-1}_0 (X)$ and~$\BH_p^m (X)$ is given by the transgression. As a consequence, we have 
\begin{equation}\label{eq:dominationchi}
(-1)^{m-1}D_c(\tG)(-1)= {\chi}(X) - 1\, ,
\end{equation}
and the statement follows. 
\end{proof}

Note that the $1$-Leray assumption in Theorem~\ref{cor:chordalchar2} is essential. Indeed, consider the simplicial complex of Example~\ref{ex:sqdiags} shown in Figure~\ref{fig:squarediags}. The connected domination polynomial of the underlying graph, evaluated at $-1$, is $-1$, whereas the Euler characteristic of $X$ is $1$. 

Corollary~\ref{cor:chordalchar} follows immediately from Theorem~\ref{cor:chordalchar2}, after observing that flag complexes of chordal graphs are contractible and $1$-Leray~\cite[Lemma~3.1]{ChordalContractible}. 

\begin{rem}
Consider the graph $\tI_3\square \tI_2$. By Corollary~\ref{cor:grids}, its bold homology is trivial, and $-1$ is a root of its connected domination polynomial. Hence, the converse of Corollary~\ref{cor:chordalchar} does not generally hold, as $\tI_3\square \tI_2$ is not chordal.
\end{rem}

We conclude with a perspective on higher chordality.
From \cite[Fact 3.1]{chordality}, the anti-star cover $\cU^{\mathrm{ast}}$ of a simplicial complex $X$ is $k$-Leray, for $k>1$, if and only if $X$ is resolution $l$-chordal for all $l\geq k$. Paralleling Corollary~\ref{cor:chordalchar}, it would be interesting to investigate properties of the connected domination polynomial of the $1$-skeleton of such complexes.

\end{document}